\documentclass[11pt]{amsart}
\usepackage{amsmath}
\usepackage{amssymb}
\usepackage{comment}
\usepackage{color}
\usepackage{graphicx}
\definecolor{dred}{rgb}{0,0,0.6}

\newtheorem{thm}{Theorem}[section]
\newtheorem{cor}[thm]{Corollary}
\newtheorem{lem}[thm]{Lemma}

\theoremstyle{definition}

\theoremstyle{remark}
\newtheorem{rem}[thm]{\bf{Remark}}
\numberwithin{equation}{section}

\newcommand{\beas}{\begin{eqnarray*}}
\newcommand{\eeas}{\end{eqnarray*}}
\newcommand{\bes} {\begin{equation*}}
\newcommand{\ees} {\end{equation*}}
\newcommand{\be} {\begin{equation}}
\newcommand{\ee} {\end{equation}}
\newcommand{\bea} {\begin{eqnarray}}
\newcommand{\eea} {\end{eqnarray}}
\newcommand{\ra} {\rightarrow}
\newcommand{\txt} {\textmd}

\newcommand{\R}{\mathbb R}
\newcommand{\C}{\mathbb C}

\newcommand{\N}{\mathbb N}

\newcommand{\la}{\lambda}
\newcommand{\g}{\mathfrak{g}}

\begin{document}

\title[Extension problem and Hardy's inequality]{An extension problem and Hardy's inequality for the fractional Laplace-Beltrami operator on Riemannian symmetric spaces of noncompact type}

\author{Mithun Bhowmik and Sanjoy Pusti}

\address{(Mithun Bhowmik) Department of Mathematics, Indian Institute of Science, Bangalore-560012, India}
\email{mithunb@iisc.ac.in, mithunbhowmik123@gamil.com}

\address{(Sanjoy Pusti) Department of Mathematics, IIT Bombay, Powai, Mumbai-400076, India}
\email{sanjoy@math.iitb.ac.in}



\begin{abstract}
In this paper we study an extension problem for the Laplace-Beltrami operator on Riemannian symmetric spaces of noncompact type and use the solution to prove Hardy-type  inequalities for fractional powers of the Laplace-Beltrami operator. Next, we study the mapping properties of the extension operator. In the last part we prove Poincar\'e-Sobolev inequalities on these spaces.
\end{abstract}



\maketitle


\section{Introduction}
In recent years there has been intensive research on various kinds of inequalities for fractional order operators because of their applications to many areas of analysis (see for instance \cite{BFL, FLS, ST} and the references therein). The classical definitions of the fractional operator in terms of the Fourier analysis involve functional analysis and singular integrals. They are nonlocal objects. This fact does not allow to apply local PDE techniques to treat
nonlinear problems for the fractional operators. To overcome this difficulty, in the Euclidean case, Caffarelli and Silvestre \cite{Caffa} studied the extension problem associated to the Laplacian and realised the fractional power as the map taking Dirichlet data to the Neumann data. On a certain class of noncompact manifolds, this definition of the fractional Laplacian through an extension problem has been studied by Banika et al. \cite{BGS}.

In the first part of this article we will concern with the Hardy-type inequalities for the fractional operators. Let $\Delta_{\R^n} =\sum_{j=1}^n\frac{\partial^2}{\partial x_j^2}$
denote the Euclidean Laplacian on $\R^n$. For $0<s< n/2$ and $f\in C_c^\infty(\R^n)$, the Hardy's inequality for fractional powers of the Laplacian states the following
\be \label{fhrn}
\int_{\R^n}\frac{|f(x)|^2}{|x|^{2s}}\,dx\leq 4^{-s}\frac{\Gamma\left(\frac{n-2s}{4}\right)^2}{\Gamma\left(\frac{n+2s}{4}\right)^2} \langle (-\Delta_{\R^n})^s f, f\rangle.
\ee
This is a generalization of the original Hardy's inequality proved for the gradient $\nabla_{\R^n}$ of $f$: for $n \geq 3$,
\be\label{eqn-hardy}
\frac{(n-2)^2}{4} \int_{\R^n} \frac{|f(x)|^2}{|x|^2} dx \leq \int_{\R^n} |\nabla_{\R^n} f(x)|^2~ dx, \:\: \textit{ for } f\in C_c(\R^n).
\ee
The constant appearing in the equation (\ref{fhrn}) is sharp \cite{B, Hr, Y}. It is also known that the equality is
not obtained in the class of functions for which both sides of the inequality (\ref{fhrn}) are finite. Using a ground state representation, Frank, Lieb, and Seiringer gave a different proof of the inequality (\ref{fhrn})
when $0 < s < \min\{1, n/2\}$ which improved the previous results \cite{FLS}. There is another version of Hardy's inequality where the homogeneous weight function $|x|^{-2s}$ is replaced by non-homogeneous one:
\be \label{fhrn1}
\int_{\R^n}\frac{|f(x)|^2}{(\delta^2+|x|^2)^{2s}}~dx\leq 4^{-s}\frac{\Gamma\left(\frac{n-2s}{4}\right)}{\Gamma\left(\frac{n+2s}{4}\right)} ~\delta^{-2s}~\langle (-\Delta_{\R^n})^s f, f\rangle, \:\: \delta>0.
\ee
Here also the constant is sharp and equality is achieved for the functions $(\delta^2 + |x|^2)^{-(n-2s)/2}$ and their translates \cite{BRT}.

Generalization of the classical Hardy's inequality (\ref{eqn-hardy}) to Riemannian
manifolds was intensively pursued after the seminal work of Carron \cite{C}, see for instance \cite{BGG, DD, KO, KO1, K, YSK}. In \cite{C}, the following weighted Hardy's inequality was obtained on a complete noncompact Riemannian manifold $M$:
\bes
\int_{M} \eta^\alpha |\nabla_{g} \phi|^2~dv_g \geq \left(\frac{C+\alpha-1}{2}\right)^2\int_{M} \eta^\alpha \frac{\phi^2}{\eta^2}~dv_g,
\ees 
where $\phi\in C_c^\infty(M- \eta^{-1}\{0\}), ~ \alpha\in \R, ~C>1,~ C+\alpha-1>0$ and the weight
function $\eta$ satisfies $|\nabla_M \eta| = 1$ and $|\Delta_M \eta| \geq C/\eta$ in the sense of distribution. Here $\nabla_g, dv_g$ denote respectively the Riemannian gradient and Riemannian measure on $M$. In the case of Cartan-Hadamard manifold $M$ of dimension $N$ (namely, a manifold which is complete, simply-connected, and has everywhere non-positive sectional curvature), the geodesic distance function $d(x, x_0)$, where $x_0 \in M$, satisfies all the assumptions of the weight $\eta$ and the above inequality holds with the best constant $(N-2)^2/4$, see \cite{K}. Analogues of Hardy-type inequalities for fractional powers of the sublaplacian are also known, for instance, the work by P. Ciatti, M. Cowling and F. Ricci for stratified Lie groups \cite{CCR}. There the authors have not paid attention to the sharpness of the constants. Recently, in \cite{RT2}, Roncal and Thangavelu have proved analogues of Hardy-type inequalities with  sharp constants for fractional powers of the sublaplacian on the Heisenberg group. For recent results on the Hardy-type inequalities for the fractional operators we refer \cite{BRT, RT1, RzS}.

Our first aim in this article is to prove analogues of Hardy's inequalities (\ref{fhrn}) and (\ref{fhrn1}) for fractional powers of the Laplace-Beltrami operator $\Delta$ on Riemannian symmetric space $X$ of noncompact type. We have the following analogue of Hardy's inequality in the non-homogeneous case.
\begin{thm} \label{Hardy-inhomogeneous}
Let $0<\sigma<1$ and $y>0$. Then there exists a constant $C_\sigma>0$ such that for $F\in H^\sigma(X)$
\bes
\left\langle (-\Delta)^\sigma F, F\right\rangle \geq   C_\sigma\, y^{2\sigma} \left( \int_{\left\{x:|x|^2+y^2<1\right\}} \frac{|F(x)|^2}{(y^2 +|x|^2)^{2\sigma}}\,dx + \int_{\left\{x:|x|^2+y^2\geq 1\right\}} \frac{|F(x)|^2}{(y^2 +|x|^2)^{\sigma}}\,dx\right).
\ees
\end{thm}
\begin{rem}
In contrast with the inequality (\ref{fhrn1}) for the Euclidean space, we get an improvement in the theorem above. This comes as a consequence of the geometry of the symmetric space. In the following theorem also we get similar improvement.
\end{rem}
For the homogeneous weight function, we prove the following analogue of Hardy's inequality on $X$.
\begin{thm}\label{Hardy-homogeneous}
Let $0<\sigma<1$. Then there exists a constant $C'_\sigma>0$ such that for $F\in C_c^\infty(X)$ 
\bes
\left\langle (-\Delta)^\sigma F, F\right\rangle \geq C'_\sigma  \left( \int_{\left\{x:|x|<1\right\}} \frac{|F(x)|^2}{|x|^{2\sigma}}\,dx + \int_{\left\{x:|x|\geq 1\right\}} \frac{|F(x)|^2}{|x|^{\sigma}}\,dx\right).
\ees
\end{thm}
Given $\sigma \in (0, 1)$, the fractional Laplacian $(-\Delta_{\R^n})^\sigma$ on $\R^n$ is defined as a pseudo-differential operator by
\bes
\mathcal F\left((-\Delta_{\R^n})^\sigma f\right) (\xi)=|\xi|^{2\sigma} \mathcal F f(\xi), \:\: \xi\in \R^n,
\ees
where $\mathcal Ff$ is the Fourier transform of $f$ given by
\bes
\mathcal Ff(\xi)=(2\pi)^{-n/2}\int_{\R^n} f(x)~e^{-i x\cdot \xi}~dx, \:\: \xi\in \R^n.
\ees
It can also be written as the singular integral
\bes
(-\Delta_{\R^n})^\sigma f (x)= c_{n, \sigma} P.V.\int_{\R^n} \frac{f(x)-f(y)}{|x-y|^{n+2\sigma}}~dy,
\ees
where $c_{n, \sigma}$ is a positive constant. Caffarelli and Silvestre have developed in \cite{Caffa} an equivalent definition of the fractional Laplacian $(-\Delta_{\R^n})^\sigma, \sigma \in (0, 1)$, using an extension problem to the upper half-space $\R^{n+1}_+$. For a function $f: \R^n \ra \R$, consider the solution $u: \R^n \times [0, +\infty) \ra \R$ of the following differential equation
\bea \label{extn-rn}
&&\Delta_{\R^n} u + \frac{(1-2\sigma)}{y}\frac{ \partial u}{\partial y} +\frac{ \partial^2 u}{\partial y^2}=0, \:\: y>0;\\
&&  u(x, 0)=f(x),\:\:\:\: x\in \R^n \nonumber.
\eea
Then the fractional Laplacian of $f$ can be computed as
\bes
(-\Delta_{\R^n})^\sigma f= -2^{2\sigma-1} \frac{\Gamma(\sigma)}{\Gamma(1-\sigma)} \lim_{y\ra 0^+} y^{1-2\sigma}\frac{\partial u}{\partial y}.
\ees
The Poisson kernel for the fractional Laplacian $(-\Delta_{\R^n})^\sigma $
in $\R^n$ is
\bes
K_{\sigma}(x, y)= c_{n, \sigma}\frac{y^{2\sigma}}{(|x|^2+y^2)^{\sigma+\frac{n}{2}}},
\ees
and then $u(x, y) =  f\ast_{\R^n} K_\sigma$. Therefore 
\bes
(-\Delta_{\R^n})^\sigma f= -2^{2\sigma-1} \frac{\Gamma(\sigma)}{\Gamma(1-\sigma)} \lim_{y\ra 0^+} y^{1-2\sigma}\frac{\partial}{\partial y} (f \ast_{\R^n} K_\sigma)(x).
\ees
Later, Stinga and
Torrea \cite{ST} showed that one can define the fractional Laplacian on a domain $\Omega  \subset \R^n$
through the extension (\ref{extn-rn}) using the heat-diffusion semigroup generated by the Laplacian $\Delta_\Omega$ provided that the heat kernel associated with $\Delta_\Omega$ exists
and it satisfies some decay properties. Since the heat kernel on general noncompact manifolds has been extensively studied depending on the underlying geometry, Banica et al. in \cite{BGS} take this approach to define the fractional Laplace-Beltrami operator on some noncompact manifolds which in particular, include the Riemannian symmetric spaces of noncompact type. Let ${\bf d}$ be a Riemannian metric on a Riemannian symmetric space $X$ and $\Delta$ be the corresponding Laplace-Beltrami operator on $X$. Also, let ${\bf g}$ be the product metric on $X \times \R^+$ given by ${\bf g} = {\bf d} + dy^2$. For $\sigma>0$, let $H^\sigma(X)$ denote the Sobolev space on $X$ (defined in Section 2). In \cite[Theorem 1.1]{BGS}, the following result is proved for the Riemannian symmetric space $X$ of noncompact type of arbitrary rank.
\begin{thm}$($Banica; Gon\'zalez; S\'aez$)$
Let $\sigma \in (0, 1)$. Then for any given $f \in H^\sigma(X)$, there exists a unique solution of the extension problem
\bea \label{extn1}
&&\Delta u + \frac{(1-2\sigma)}{y}\frac{ \partial u}{\partial y} +\frac{ \partial^2 u}{\partial y^2}=0, \:\: y>0;\\
&&  u(x, 0)=f(x),\:\:\:\: x\in X \nonumber. 
\eea
Moreover, the fractional Laplace-Beltrami operator on $X$ can be recovered through
\be \label{u-f-reln}
(-\Delta)^\sigma f(x)=-2^{2\sigma-1}\frac{\Gamma(\sigma)}{\Gamma(1-\sigma)}\lim_{y\ra 0^+}y^{1-2\sigma} \frac{\partial u}{\partial y}(x,y).
\ee
\end{thm} 
The following theorem gives an alternative expression of a solution of the extension problem  (\ref{extn1}), which will be useful for us. The  proof is similar to \cite[Theorem 1.1]{ST}. See also \cite[Theorem 3.1]{BGS} . For the sake of completeness we give a proof in section $3$. 
\begin{thm}\label{soln-extn-possion}
Let $f\in Dom(-\Delta)^\sigma$. A solution of (\ref{extn1}) is given by
\be \label{expression-u-1}
u(x,y)=\frac{1}{\Gamma(\sigma)}\int_{0}^{\infty}e^{t\Delta}(-\Delta)^{\sigma}f(x) e^{-y^2/4t}~\frac{dt}{t^{1-\sigma}},
\ee
and $u$ is related to $(-\Delta)^\sigma f$ by the equation (\ref{u-f-reln}). Moreover, the following Poisson formula for $u$ holds:
\be \label{u-conv-f}
u(x, y)=\int_X f(\zeta)P_y^\sigma ( \zeta^{-1} x)~d\zeta=(f \ast P_y^\sigma)(x),
\ee
where
\be \label{defn-P}
P_y^\sigma(x)=\frac{y^{2\sigma}}{4^\sigma \Gamma(\sigma)}\int_{0}^{\infty} h_t(x)~e^{-y^2/4t}~\frac{dt}{t^{1+\sigma}}.
\ee
\end{thm}
All these identities in theorem above are to be understood in the $L^2$ sense. The mapping properties of the Poisson operator $P_\sigma$ on $\R^n$ which maps boundary value $f$ to the solution $u$ of the extension problem (\ref{extn-rn}) were studied by M\"ollers et al. \cite{MOZ}. In the same paper, the authors have also obtained a similar result for Heisenberg groups. On the Euclidean spaces, they proved the following
\begin{thm}[M\"ollers; $\O$rsted;  Zhang]\label{mapping-property-real}
Let $0< \sigma< \frac{n}{2}$. Then
\begin{enumerate}
\item $P_\sigma: H^\sigma(\R^{n}) \ra H^{\sigma+1/2}(\R^{n}\times \R^+)$ is isometric up to a constant.
\item $P_\sigma$ extends to a bounded operator from $L^p(\R^n)$ to $L^q(\R^n\times \R_+)$, for $1<p\leq \infty$ and $q=\frac{n+1}{n}p$ (Figure 1, (a)). 
\end{enumerate}
\end{thm}
In \cite{Chen}, Chen proved that for particular values $p=\frac{2n}{n-2\sigma}$ and $q=\frac{2n+2}{n-2\sigma}$, there exists a sharp constant $C$ such that 
\bes
\|P_\sigma f\|_{L^q(\R^n)}\leq C\|f\|_{L^p(\R^n)}, \:\: \textit{ for } f\in L^p(\R^n),
\ees
and the optimizer of this inequality are translations, dilations and multiples of the function
\begin{equation}
\nonumber
f(x)=\left(1+|x|^2\right)^{-\frac{n}{2} +\sigma}.
\end{equation} 

Our second main aim in this article is to study the mapping properties of the ``Poisson operator'' $T_\sigma$ given by 
\be \label{defn-extn}
T_\sigma f(x, y)= f\ast P_y^\sigma, \:\:  x\in X,\: y>0,
\ee
which maps $f$ to the solution $u$ of the extension problem (\ref{extn1}) related to the Laplace-Beltrami operator on Riemannian symmetric spaces of noncompact type. The following analogue of Theorem \ref{mapping-property-real} is our main result in this direction. 
\begin{thm}\label{mapping-property}
Let $\dim X=n$ and $0< \sigma< 1$. Then
\begin{enumerate}
\item $T_\sigma: H^\sigma(X) \ra H^{\sigma+1/2}(X\times \R_+)$ is isometric up to a constant.
		
\item $T_\sigma$ extends to a bounded operator from $L^p(X)$ to $L^q(X\times \R_+)$, for $1< p <\infty$ and $p<q\leq\frac{n+1}{n}p$; and from $L^1(X)$ to $L^q(X)$, for $1<q< \frac{n+1}{n}$ (Figure 1, (b)).
\end{enumerate}
\end{thm}

\begin{figure}[ht]
{\centering{\includegraphics{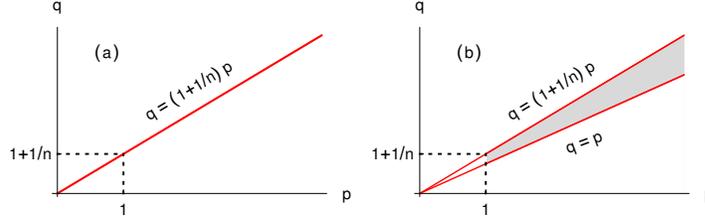}}\par}
\label{fig1}
\caption{(a) Euclidean
(b) Symmetric spaces}
\end{figure}

\begin{rem} \label{rem-k-s}
In contrast with Theorem \ref{mapping-property-real} on Euclidean space, the exponents $p, q$ in Theorem \ref{mapping-property} on $X$ can vary over a much larger region (see in the figure 1 above). This striking phenomenon comes as a consequence of the Kunze-Stein phenomenon. The Kunze-Stein phenomenon, proved by Cowling \cite{Cowling} on connected semi-simple Lie groups $G$ with finite center, says that the convolution inequality
\bes
L^2(G) \ast L^p(G)\subset L^2(G),
\ees
holds for $p \in [1, 2)$. We note that above inequalities on Euclidean space are only valid for $p = 1$. We use the following generalize version \cite[Theorem 2.2, (ii)]{CGM}: let $k\in L^q(X)$, for $1< q\leq 2$ and let $1\leq p<q$. Then the map $f \mapsto f \ast k$ is bounded from $L^p(X)$ to $L^q(X)$.
\end{rem}

An explicit expression of the heat
kernel is known for certain symmetric spaces. Using this in section 5, we write the precise expression of the kernel $P_y^\sigma$ in the case of complex and rank one symmetric spaces. 

The final topic we shall deal with here is analogues of the Poincar\'e-Sobolev inequalities for the fractional Laplace-Beltrami operator on $X$. In \cite{MS},
Mancini and Sandeep proved the following 
optimal Poincar\'{e}-Sobolev inequalities for the Laplace-Beltrami operator $\Delta_{\mathbb H^n}$ on the real hyperbolic space $\mathbb H^n$ of dimension $n\geq 3$. 
\begin{thm} $($Mancini; Sandeep$)$\label{thm-vor}
Let $n\geq 3$. Then for $2<p\leq \frac{2n}{n-2}$, there exists $S=S_{n, p}>0$ such that for all $u\in C_c^\infty(\mathbb H^n)$, 
\bes
\|\left(-\Delta_{\mathbb H^n}-(n-1)^2/4)\right)^{1/2}u\|_{L^2(\mathbb{H}^n)}^2 \geq S\|u\|_{L^p(\mathbb{H}^n)}^2.
\ees
\end{thm}
In case of real hyperbolic space $\mathbb H^3$ of dimension three, Benguria, Frank and Loss \cite{BFL} proved that the best constant $S_3$ in the theorem above is the same as the best sharp Sobolev constant for the first order Sobolev inequality on $\mathbb H^3$. Recently, using Green kernel estimates Li, Lu, Yang \cite[Theorem 6.2]{Li-Lu} proved the following Poincar\'{e}-Sobolev inequalities for the fractional Laplace-Beltrami operator $\Delta_{\mathbb H^n}$ on $\mathbb H^n$.
\begin{thm} $($Li; Lu; Yang$)$\label{thm-p-s}
Let $n \geq 3$ and $1 \leq \sigma < 3$. Then there exists a constant $C=C_{n, \sigma, p} > 0$ such that
\bes
\|\left(-\Delta_{\mathbb H^n}-(n-1)^2/4 \right)^{\frac{\sigma}{4}}u\|_{L^2(\mathbb H^n)}^2\geq C\|u\|^2_{L^{\frac{2n}{n-\sigma}}(\mathbb H^n)},\:\: \textit{ for } u\in H^{\frac{\sigma}{2}}(\mathbb H^n).
\ees
\end{thm}
For related results  and their sharpness, we refer the reader to 
\cite{Lu-Yang, Varo}. Our aim in the final section is to prove an  analogue of the Poincar\'{e}-Sobolev inequality for the fractional Laplace-Beltrami operator $\Delta$ on $X$ which generalizes the above mentioned theorems. The idea of the proof is to use the estimate of the Bassel-Green-Riesz kernel due to Anker-Ji \cite{AJ}. Since we are working on general Riemannian symmetric spaces of noncompact type, it is difficult to get the explicit values of the constants involve and we do not make attempt to get the optimal constant. Here is our final result. We refer the reader to the next section for the unexplained notation used in the theorem below.
\begin{thm} \label{thm-p-s-X}
Let $\dim X=n\geq 3$ and $0<\sigma< \min\{l+2|\Sigma_0^+|, n\}$. Then for $2< p\leq \frac{2n}{n-\sigma}$ there exists $S=S_{n, \sigma, p}>0$ such that for all $u\in H^{\frac{\sigma}{2}}(X)$, 
\bes
\|(-\Delta-|\rho|^2)^{\sigma/4}u\|_{L^2(X)}^2 \geq S\|u\|_{L^p(X)}^2.
\ees
\end{thm}

\section{Preliminaries}
In this section, we describe the necessary preliminaries regarding semisimple Lie groups and harmonic analysis on Riemannian symmetric spaces. These are standard and can be found, for example, in \cite{GV, H, H1, H2}. To make the article self-contained, we shall gather only those results which will be used throughout this paper.

\subsection{Notations}
Let $G$ be a connected, noncompact, real semisimple Lie group with finite centre and $\mathfrak g$ its Lie algebra. We fix a Cartan involution $\theta$ of $\mathfrak g$ and write $\mathfrak g = \mathfrak k \oplus \mathfrak p$ where $\mathfrak k$ and $\mathfrak p$ are $+1$ and $-1$ eigenspaces of $\theta$ respectively. Then $\mathfrak k$ is a maximal compact subalgebra of $\mathfrak g$ and $\mathfrak p$ is a linear subspace of $\mathfrak g$. The Cartan involution $\theta$ induces an automorphism $\Theta$ of the group $G$ and $K=\{g\in G\mid \Theta (g)=g\}$ is a maximal compact subgroup of $G$. Let $B$ denote the Cartan Killing form of $\mathfrak g$. It is known that $B\mid_{\mathfrak p\times\mathfrak p}$ is positive definite and hence induces an inner product and a norm $\| \cdot \|_B$ on $\mathfrak p$. The homogeneous space $X=G/K$ is a smooth manifold. The tangent space of $X$ at the point $o=eK$ can be naturally identified to $\mathfrak p$ and the restriction of $B$ on $\mathfrak p$ then induces a $G$-invariant Riemannian metric $ {\bf d}$ on $X$. For $x\in X$ and $r>0$, we denote ${\bf B}(x, r)$ to be the ball of radius $r$ centered at $x$ in this metric.

Let $\mathfrak a$ be a maximal subalgebra in $\mathfrak p$; then $\mathfrak a$ is abelian. We assume that $\dim \mathfrak a = l$, called the real rank of $G$.
We can identify $\mathfrak a$ endowed with the inner product induced from $\mathfrak p$ with $\mathbb{R}^d$ and let $\mathfrak{a}^*$ be the real dual of $\mathfrak{a}$. The set of restricted roots of the pair $(\mathfrak g, \mathfrak{a})$ is denoted by $\Sigma$.  It consists of all $\alpha \in \mathfrak{a}^*$ such that
\bes
\g_\alpha = \left\{X\in \g ~|~ [Y, X] = \alpha(Y) X, \:\: \txt{ for all } Y\in \mathfrak{a} \right\}
\ees
is nonzero with $m_\alpha = \dim(\g_\alpha)$. We choose a system of positive roots $\Sigma^+$ and with respect to $\Sigma^+$, the positive Weyl chamber
$\mathfrak{a}_+ = \left\{X\in \mathfrak{a} ~|~ \alpha(X)>0,\:\:  \txt{ for all } \alpha \in \Sigma^+\right\}$. We also let $\Sigma_0^+$ be the set of positive indivisible roots.
We denote by 
\bes
\mathfrak{n}= \oplus_{\alpha \in \Sigma^+}  ~ \mathfrak{g}_{\alpha}.
\ees 
Then $\mathfrak{n}$ is a nilpotent subalgebra of $\g$ and we obtain the Iwasawa decomposition $\g = \mathfrak{k} \oplus \mathfrak{a} \oplus \mathfrak{n}$. If $N=\exp \mathfrak{n}$ and $A= \exp \mathfrak{a}$ then $N$ is a Nilpotent Lie group and $A$ normalizes $N$. For the group $G$, we now have the Iwasawa decomposition 
$G= KAN$, that is, every $g\in G$ can be uniquely written as 
\bes
g=\kappa(g)\exp H(g)\eta(g), \:\:\:\: \kappa(g)\in K, H(g)\in \mathfrak{a}, \eta(g)\in N,
\ees 
and the map 
\bes
(k, a, n) \mapsto kan
\ees 
is a global diffeomorphism of $K\times A \times N$ onto $G$. Let $n$ be the dimension of $X$ then 
\bes
n=l+ \sum_{\alpha\in \Sigma^+} m_\alpha.
\ees 
We always assume that $n\geq 2$. Let $\rho$ denote the half sum of all positive roots counted with their
multiplicities:
\bes
\rho=\frac{1}{2}\sum_{\alpha\in \Sigma^+}m_{\alpha}~\alpha.
\ees 
It is known that the $L^2$-spectrum of the Laplace-Beltrami operator $\Delta$ on $X$ is the half-line $(-\infty, -|\rho|^2]$. Let $M'$ and $M$ be the normalizer and centralizer of $\mathfrak{a}$ in $K$ respectively.
Then $M$ is a normal subgroup of $M'$ and normalizes $N$. The quotient group $W = M'/M$ is a finite group, called the Weyl group of the pair $(\g, \mathfrak{k})$. The Weyl group $W$ acts on $\mathfrak{a}$ by the adjoint action. It is known that $W$ acts as a group of orthogonal transformations (preserving the Cartan-Killing form) on $\mathfrak{a}$. Each $w\in W$ permutes the Weyl chambers and the action of $W$ on the Weyl chambers is simply transitive. Let $A_+= \exp{\mathfrak{a_+}}$. Since $\exp: \mathfrak{a} \to A$ is an isomorphism we can identify $A$ with $\R^d$. If $\overline{A_+}$ denotes the closure of $A_+$ in $G$, then one has the polar decomposition $G=K A K$,
that is, each $g\in G$ can be written as 
\bes
g=k_1 (\exp Y) k_2, \:\:  k_1, k_2 \in K, Y\in \mathfrak{a}.
\ees 
In the above decomposition, the $A$ component of $\mathfrak{g}$ is uniquely determined modulo $W$. In particular, it is well defined in $\overline{A_+}$. The map $(k_1, a, k_2)\mapsto k_1ak_2$ of $K\times A\times K$ into $G$ induces a diffeomorphism of $K/M\times A_+\times K$ onto an open dense subset of $G$. We extend the inner product on $\mathfrak{a}$ induced by $B$ to $\mathfrak{a}^*$ by duality, that is, we set
\bes
\langle \la, \mu \rangle =B(Y_\la, Y_\mu), \:\:\:\: \la, \mu \in \mathfrak{a}^*,  ~ Y_\la, Y_\mu \in \mathfrak{a},
\ees
where $Y_\la$ is the unique element in $\mathfrak{a}$ such that 
\bes
\la(Y) = B(Y_\la, Y), \:\:\:\: \txt{ for all } Y\in \mathfrak{a}.
\ees
This inner product induces a norm, again denoted by $|\cdot|$, on $\mathfrak{a}^*$,
\bes
|\la| = \langle \la, \la \rangle^{\frac{1}{2}}, \:\:\:\: \la \in \mathfrak{a}^*.
\ees
The elements of the Weyl group $W$ acts on $\mathfrak a^*$ by the formula
\bes
sY_{\la}=Y_{s\la},\:\:\:\:\:\:s\in W,\:\la\in\mathfrak a^*.
\ees
Let $\mathfrak{a}_\C^*$ denote the complexification of $\mathfrak{a}^*$, that is, the set of all complex-valued real linear functionals on $\mathfrak{a}$. The usual extension of $B$ to $\mathfrak{a}_\C^*$, using conjugate linearity is also denoted by $B$. Through the identification of $A$ with $\R^d$, we use the Lebesgue measure on $\R^d$ as the Haar measure $da$ on $A$. As usual on the compact group $K$, we fix the normalized Haar measure $dk$ and $dn$ denotes a Haar measure on $N$. The following integral formulae describe the Haar measure of $G$ corresponding to the Iwasawa and polar decomposition respectively.
For any $f\in C_c(G)$,
\beas
\int_{G}{f(g)dg} &=& \int_K \int_{\mathfrak{a}}\int_N f(k\exp Y n)~e^{2\rho(Y)}~dn~dY~dk \\ 
&=&\int_{K}{\int_{\overline{A_+}}{\int_{K}{f(k_1ak_2) ~ J(a)~dk_1~da~dk_2}}},
\eeas
where $dY$ is the Lebesgue measure on $\R^d$ and for $H\in \overline{\mathfrak{a}_+}$
\be \label{J-est}
J(\exp H)= c \prod_{\alpha\in \Sigma^+}\left(\sinh\alpha(H)\right)^{m_{\alpha}}  \asymp \left\{\prod_{\alpha\in \Sigma^+}\left(\frac{\alpha(H)}{1 + \alpha(H)}\right)^{m_\alpha} \right\} e^{2\rho(H)},
\ee
where $c$ (in the equality above) is a normalizing constant. 
If $f$ is a function on $X= G/K$ then $f$ can be thought of as a function on $G$ which is right invariant under the action of $K$. It follows that on $X$ we have a $G$ invariant measure $dx$ such that 
\be \label{integral-X}
\int_X f(x)~dx= \int_{K/M}\int_{\mathfrak{a}_+}f(k\exp Y)~J(\exp Y)~dY~dk_M,
\ee
where $dk_M$ is the $K$-invariant measure on $K/M$.  

\subsection{Fourier analysis on $X$}
For a sufficiently nice function $f$ on $X$, its Fourier transform $\widetilde{f}$ is a function defined on $\mathfrak{a}_{\C}^* \times K$ given by 
\bes 
\widetilde{f}(\la,k) = \int_{G} f(g) e^{(i\la - \rho)H(g^{-1}k)} dg,\:\:\:\:\:\: \la \in \mathfrak{a}_{\C}^*,\:\: k \in K, 
\ees
whenever the integral exists \cite[P. 199]{H1}. 
As $M$ normalizes $N$ the function $k\mapsto\widetilde{f}(\la, k)$ is right $M$-invariant.
It is known that if $f\in L^1(X)$ then $\widetilde{f}(\la, k)$ is a continuous function of $\la \in \mathfrak{a}^*$, for almost every $k\in K$ (in fact, holomorphic in $\lambda$ on a domain containing $\mathfrak a^\ast$). If in addition, $\widetilde{f}\in L^1(\mathfrak{a}^*\times K, |{\bf c}(\la)|^{-2}~d\la~dk)$ then the following Fourier inversion holds,
\bes
f(gK)= |W|^{-1}\int_{\mathfrak{a}^*\times K}\widetilde{f}(\la, k)~e^{-(i\la+\rho)H(g^{-1}k)} ~ |{\bf c}(\la)|^{-2}d\la~dk,
\ees
for almost every $gK\in X$ \cite[Chapter III, Theorem 1.8, Theorem 1.9]{H1}. Here ${\bf c}(\la)$ denotes Harish Chandra's ${\bf c}$-function. Moreover, $f \mapsto \widetilde{f}$ extends to an isometry of $L^2(X)$ onto $L^2(\mathfrak{a}^*_+\times K, |{\bf c}(\la)|^{-2}~d\la~dk )$ \cite[Chapter III, Theorem 1.5]{H1}:
\bes
\int_X |f(x)|^2~dx= |W|^{-1} \int_{\mathfrak a^\ast \times K} |\widetilde f(\lambda, k)|^2~|{\bf c}(\lambda)|^{-2}~d\lambda~dk.
\ees
It is known \cite[Ch. IV, prop 7.2]{H2} that there exists a positive number $C$ and $d\in \N$ such that for all $\lambda\in\mathfrak a_+^*$ 
\bea \label{clambdaest}
|{\bf c}(\la)|^{-2} &\leq& C(1+ |\la|)^{n-l}, \:\: \textit{ for } |\lambda|\geq 1;\\
&\leq& C(1+ |\la|)^{d}, \:\: \textit{ for } |\lambda|<1 \nonumber.
\eea
We now specialize in the case of $K$-biinvariant function $f$ on $G$. Using the polar decomposition of $G$ we may view a $K$-biinvariant integrable function $f$ on $G$ as a function on $A_+$, or by using the inverse exponential map we may also view $f$ as a function on $\mathfrak{a}$ solely determined by its values on $\mathfrak{a}_+$. Henceforth, we shall denote the set of $K$-biinvariant functions in $L^1(G)$ by $L^1(K \backslash G/K)$.
If $f\in L^1(K \backslash G/K)$ then the Fourier transform $\widetilde{f}$ reduces to the  spherical Fourier transform $\widehat f(\la)$ which is given by the integral 
\be \label{defn-sft}
\widetilde{f}(\lambda, k)=\widehat f(\la):= \int_Gf(g)\phi_{-\la}(g)~dg,
\ee 
for all $k\in K$ where 
\be \label{philambda} 
 \phi_\la(g) 
= \int_K e^{-(i\la+ \rho) \big(H(g^{-1}k)\big)}~dk,\:\:\:\:\:\:\la \in \mathfrak{a}_\C^*,  
\ee
is Harish Chandra's elementary spherical function.
We now list down some well-known properties of the elementary spherical functions which are important for us (\cite[Prop. 2.2.12]{AJ},  \cite[Prop. 3.1.4]{GV}; \cite[ Lemma 1.18, P. 221]{H1}).
\begin{thm} \label{thm-phi}
\begin{enumerate}
\item[(1)] $\phi_\la(g)$ is $K$-biinvariant in $g\in G$ and $W$-invariant in $\la\in \mathfrak{a}_\C^*$.
\item[(2)] $\phi_\la(g)$ is $C^\infty$ in $g\in G$ and holomorphic in $\la\in \mathfrak{a}_\C^*$.
\item[(3)] The elementary spherical function $\phi_0$ satisfies the following global estimate:
\be \label{estimate-phi0}
 \phi_0(\exp H)\asymp \left\{\prod_{\alpha\in \Sigma_0^+} \left(1 + \alpha(H)\right) \right\}  e^{-\rho(H)}, \:\:\text{ for all } H\in \overline{\mathfrak a^+}.
\ee
\item[(4)] For all $\la\in \overline{\mathfrak{a}_+^*}$ we have
\be
|\phi_\la(g)| \leq  \phi_0(g)\leq 1.\label{phi0}
\ee

\end{enumerate}
\end{thm}

\subsection{Function spaces on $X$}
For $1\leq p< \infty$ we define
\bes
L^p(X \times \R)=\left\{u~|~ \|u\|_{L^p(X \times \R)}^p:= \int_{X\times \R} |u(x, y)|^p~dx~dy < \infty\right\},
\ees
and $L^p(X \times \R_+)$ to be the subspace of $L^p(X \times \R)$ consisting of all functions $u(x, y)$ which are even in the $y$-variable. We also define $L^\infty(X \times \R^+)$ analogously. 
For $\sigma>0$, the  Sobolev space of order $\sigma$ on $X$ is defined by
\bes
H^\sigma(X)= \big\{f\in L^2(X) ~|~ \|f\|_{H^\sigma(X)}^2:= \int_{\mathfrak a^\ast \times K} |\tilde f(\lambda, k)|^2 ~(|\lambda|^2+|\rho|^2)^{\sigma}~|{\bf c}(\lambda)|^{-2}~d\lambda~dk< \infty \big\}.
\ees 
Similarly, for $\sigma>0$ we  define  $H^\sigma(X \times \R)$ as the space of all functions $u\in L^2(X \times \R)$ such that
\bes
\|u\|_{H^\sigma(X \times \R)}^2:= \int_{\R}\int_{\mathfrak a^\ast \times K} |\mathcal F\left(\tilde u(\lambda, k, \cdot)(\xi)\right)|^2 ~\left(|\lambda|^2+|\rho|^2+\xi^2\right)^{\sigma}~|{\bf c}(\lambda)|^{-2}~d\lambda~dk~d\xi< \infty,
\ees
where $\mathcal F \tilde u(\lambda, k, \cdot)(\xi)$ denotes the Euclidean Fourier transform of the function $y \mapsto \tilde u(\lambda, k, y)$ at the point $\xi\in \R$, for almost every $(\lambda, k)\in
\mathfrak a^\ast \times K$. Let $H^\sigma(X \times \R_+)$ be the subspace of $H^\sigma(X \times \R)$ consisting of all elements $u(x, y)$ which are even in the $y$-variable.

\subsection{Heat kernel on $X$}
For the details of the heat kernel $h_t$ on $X= G/K$ we refer \cite{Anker, AJ}. It is a  family $\{h_t: ~ t>0\}$ of smooth functions with the following properties:
\begin{enumerate}
\item[(a)] $h_t \in L^p(K\backslash G /K), ~~ p\in [1, \infty]$, ~ for each $t>0$.
\item[(b)] For each $t>0$, $h_t$ is positive with 
\be\label{prop-ht}
\int_{G} h_t(g)~ dg = 1.
\ee
\item[(c)] $h_{t+s}= h_t*h_s, ~ t, s>0.$
\item[(d)] For each $f\in L^p(G/K), ~ p\in [1, \infty)$ the function  $u(x, t)= f*h_t(x)$, for $x\in X$
solves the heat equation
\beas
\Delta_x u(x, t)&=& \frac{\partial}{\partial t}u(x, t)\\
u(\cdot, t) &\ra & f \txt{ in } L^p(X), \txt{ as } t \ra 0.
\eeas
\item[(e)] The spherical Fourier transform of $h_t$ is given by 
\be \label{ht-ft}
\widehat{h_t}(\lambda) = e^{-t(|\lambda|^2 + |\rho|^2)}, \:\:\:\:  \lambda \in \mathfrak{a}^*.
\ee 
\end{enumerate} 

We need the following both side estimates of the heat kernel \cite[Theorem 3.7]{AJ}.
\begin{thm}\label{est-heatkernel-bothside}
Let $\kappa$ be an arbitrary positive number. Then there exists positive constants $C_1, C_2$ (depending on $\kappa$) such that 
\bes
C_1\leq \frac{h_t(\exp H)}{t^{-\frac{n}{2}}(1+t)^{\frac{n-l}{2}-|\Sigma_0^{+}|} \left\{\prod_{\alpha\in \Sigma_0^{+}}(1+ \alpha(H)\right\}e^{-|\rho|^2t-\rho(H)- \frac{|H|^2}{4t}}} \leq C_2,
\ees
for all $t>0$, and $H\in \overline{\mathfrak a^+}$, with $|H|\leq \kappa (1+t)$.
\end{thm}
For $H\in \overline{\mathfrak a^+}$ with $t \ll H$, we will use the following global upper bound \cite[Theorem 3.1]{Anker}
\be \label{est-ht}
|h_t(\exp H)|\leq t^{-d_1}(1+|H|)^{d_2}e^{-|\rho|^2t -\rho(H)-|H|^2/(4t)},
\ee
where $d_1$ and $d_2$ are positive constants depending on the position of $H \in \overline{\mathfrak a^+}$
with respect to the walls and on the relative size of $t > 0$ and $1 + |H|$.

\section{Extension problem and kernel estimates}
 Since we are dealing with fractional operators, it is natural to relate the fractional Laplace-Beltrami operator acting on $f$ to the solution $u$ in (\ref{extn1}). We proceed by proving Theorem \ref{soln-extn-possion} which will provide us an expression for the Poisson kernel of the extension operator. This will crucially be used throughout this paper.
\begin{proof}[Proof of Theorem \ref{soln-extn-possion}]
Using the heat-diffusion semigroup generated by the Laplace-Beltrami operator, the first part of the theorem follows exactly as in \cite[Theorem 1.1]{ST}. We will prove the second part. Let $f\in Dom(-\Delta)^\sigma$, and $u$ be the solution of the extension problem (\ref{extn1}) given by equation (\ref{expression-u-1}).  It now follows that
\begin{equation}
\nonumber
\left\langle u(\cdot, y), g\right\rangle=\frac{1}{\Gamma(\sigma)}\int_{|\rho|^2}^\infty \int_0^\infty e^{-t\lambda} \lambda^\sigma e^{-y^2/4t} \frac{dt}{t^{1-\sigma}}\, dE_{f, g}(\lambda),
\end{equation}
for all $g\in L^2(X)$, where  $dE_{f, g}(\lambda)$ is the regular Borel complex measure of bounded variation concentrated on the spectrum $[|\rho|^2, \infty)$ of $-\Delta$ with $d|E_{f, g}|(|\rho|^2, \infty) \leq \|f\|_{L^2(X)} \|g\|_{L^2(X)}$.
By the Fubini's theorem, putting $r\lambda=y^2/4t$ the above equation yields 
\beas
\left\langle u(\cdot, y), g\right\rangle &=& \frac{y^{2\sigma}}{4^\sigma \Gamma(\sigma)}\int_0^\infty   \int_{|\rho|^2}^\infty e^{-y^2/4r}~e^{-r\lambda} \,dE_{f, g}(\lambda) \frac{dr}{r^{1+\sigma}}\\
&=& \frac{y^{2\sigma}}{4^\sigma \Gamma(\sigma)}\int_0^\infty \left\langle e^{r\Delta} f, g\right\rangle_{L^2(X)} e^{-y^2/4r} \frac{dr}{r^{1+\sigma}} \\ \\
&=& \left\langle \frac{y^{2\sigma}}{4^\sigma \Gamma(\sigma)}\int_0^\infty e^{r\Delta}f~ e^{-y^2/4r} \frac{dr}{r^{1+\sigma}}, g\right\rangle.
\eeas
This proves that 
\bes 
u(x, y)=\frac{y^{2\sigma}}{4^\sigma \Gamma(\sigma)}\int_0^\infty e^{r\Delta}f(x) e^{-y^2/4r} \frac{dr}{r^{1+\sigma}}.
\ees
Now, using $e^{t\Delta}f= f\ast h_t$   
and Fubini's theorem, we get from the above equation that 
\bes
u(x, y)= \frac{y^{2\sigma}}{4^\sigma \Gamma(\sigma)} \int_X \int_0^\infty f(x z^{-1})~h_t(z) e^{-y^2/4t} \,\frac{dt}{t^{1+\sigma}}~dz=f \ast P_y^\sigma(x),
\ees
where the kernel $P_y^\sigma$ is given by the equation (\ref{defn-P}).
\end{proof}
As in \cite[Theorem 2.1]{ST}, we have the following consequences of the theorem above.
\begin{cor} 
Let $u(x, y)= (f \ast P_y^\sigma)(x)$, for $x\in X, y>0$ be the solution of the extension problem (\ref{extn1}) given in Theorem \ref{soln-extn-possion}. Then
\begin{enumerate}
\item[(a)] $\sup_{y\geq 0} |u(x, y)|\leq \sup_{t\geq 0} |f \ast h_t|$ in $X$.
\item[(b)] $\|u(\cdot, y)\|_{L^p(X)}\leq \|f\|_{L^p(X)}$, ~ for all $y\geq 0$ and $p\in [1, \infty]$.
\item[(c)] $\lim_{y\ra 0^+} u(\cdot, y)=f$ in $L^p(X)$,  for $p\in [1, \infty)$.
\end{enumerate}
\end{cor}
\begin{proof}
Part (a) follows from the expression of the Poisson kernel given in equation (\ref{defn-P}). For (b), we observe that $e^{t\Delta}$ has the contraction property in $L^p(X)$. Hence, by equation (\ref{defn-P}) and Minkowski's integral inequality it follows that  
\beas
\|u(\cdot, y)\|_{L^p(X)}\leq  \frac{y^{2\sigma}}{4^\sigma \Gamma(\sigma)}\int_{0}^\infty \|f \ast h_t\|_{L^p(X)}~e^{-y^2/4t}~\frac{dt}{t^{1+\sigma}}\leq \|f\|_{L^p(X)}.
\eeas
Similarly, for part (c) we observe that
\bes
\|u(\cdot, y)- f\|_{L^p(X)}\leq  \frac{y^{2\sigma}}{4^\sigma \Gamma(\sigma)}\int_{0}^\infty \|f \ast h_t-f\|_{L^p(X)}~e^{-y^2/4t}~\frac{dt}{t^{1+\sigma}}.
\ees
Since $\|f \ast h_t-f\|_{L^p(X)}\leq 2\|f\|_{L^p(X)}$, using dominated convergence theorem the result follows from the fact  that $\lim_{t \ra 0^+} f \ast h_t= f$ in $L^p(X)$, for $p\in [1, \infty)$. 
\end{proof}

For $0< \sigma< 1$ and $y>0$, let us define the function $P_y^{-\sigma}$ given by the equation (\ref{defn-P}), that is
\bes
P_y^{-\sigma}(x)=\frac{y^{-2\sigma}}{4^{-\sigma} \Gamma(-\sigma)}\int_{0}^{\infty} h_t(x)~e^{-y^2/4t}~\frac{dt}{t^{1-\sigma}},\:\: \textit{ for } x\in X.
\ees
By the estimate of the heat kernel (Theorem \ref{est-heatkernel-bothside}), it follows that $P_y^{-\sigma}$ is well defined. For $0< \sigma<1$,  we observe that $\Gamma(-\sigma):= \frac{\Gamma(1-\sigma)}{-\sigma}<0$ and hence $P_y^{-\sigma}\leq 0$. Since the heat kernel $h_t$ is $K$-biinvariant so is the function $P_y^{-\sigma}$. By (\ref{defn-sft}) the spherical Fourier transform is given by
\be \label{p-h-ft}
\widehat{P_y^{-\sigma}}(\lambda)=\int_X P_y^{-\sigma}(x)~\phi_{-\lambda}(x)~dx
=\frac{y^{-2\sigma}}{4^{-\sigma} \Gamma(-\sigma)}\int_{0}^{\infty} \widehat{h_t}(\lambda)~e^{-y^2/4t}~\frac{dt}{t^{1-\sigma}},\:\: \textit{ for } \lambda\in \mathfrak a^\ast.
\ee
Interchange of the integration is possible by the Fubini's theorem. Indeed, by (\ref{phi0}) and (\ref{ht-ft})
\beas
\int_{0}^\infty \int_X  h_t(x)~|\phi_{-\lambda}(x)|~dx ~e^{-y^2/4t}~\frac{dt}{t^{1-\sigma}} &\leq& \int_{0}^\infty \int_X  h_t(x)~\phi_0(x)~dx ~e^{-y^2/4t}~\frac{dt}{t^{1-\sigma}}\\
&=& \int_{0}^\infty e^{-t|\rho|^2}~e^{-y^2/4t}~\frac{dt}{t^{1-\sigma}}<\infty.
\eeas 
Moreover, $P_y^{-\sigma}$ is contained in the Sobolev space $H^\sigma(X)$.  Indeed, by using (\ref{p-h-ft}), (\ref{ht-ft}) and Minkowski's integral inequality we get that  
\beas
&&\|P_y^{-\sigma}\|_{H^\sigma(X)} = \left(\int_{\mathfrak a^\ast} |\widehat{P_y^{-\sigma}}(\lambda)|^2~\left(|\lambda|^2+|\rho|^2\right)^\sigma~|{\bf c}(\lambda)|^{-2}~d\lambda\right)^{\frac{1}{2}} \\
&\leq&  \frac{y^{-2\sigma}}{4^{-\sigma} |\Gamma(-\sigma)|} \int_{0}^\infty \left(\int_{\mathfrak a^\ast}|\widehat h_t(\lambda)|^2 ~(|\lambda|^2+|\rho|^2)^\sigma~|{\bf c}(\lambda)|^{-2}~d\lambda\right)^{\frac{1}{2}} ~e^{-y^2/4t}~\frac{dt}{t^{1-\sigma}}\\
&\leq& \frac{y^{-2\sigma}}{4^{-\sigma} |\Gamma(-\sigma)|} \int_{0}^\infty \left(\int_{\mathfrak a^\ast}~e^{-t(|\lambda|^2+|\rho|^2)}~ (|\lambda|^2+|\rho|^2)^\sigma~|{\bf c}(\lambda)|^{-2}~d\lambda\right)^{\frac{1}{2}}~e^{-\frac{|\rho|^2}{2}t}~e^{-y^2/4t}~\frac{dt}{t^{1-\sigma}}\\
&=& I_1+I_2,
\eeas
where 
\be \label{defn-I-p} 
I_1=\frac{y^{-2\sigma}}{4^{-\sigma} |\Gamma(-\sigma)|} \int_{0}^1 \left(\int_{\mathfrak a^\ast}~e^{-t(|\lambda|^2+|\rho|^2)}~ (|\lambda|^2+|\rho|^2)^\sigma~|{\bf c}(\lambda)|^{-2}~d\lambda\right)^{\frac{1}{2}}~e^{-\frac{|\rho|^2}{2}t}~e^{-y^2/4t}~\frac{dt}{t^{1-\sigma}},
\ee
and $I_2$ is defiend as above with the integration in the $t$-variable over the interval $[1, \infty)$. It is enough to show that both $I_1$ and $I_2$ are finite. We consider $I_1$ first. Using the property (\ref{clambdaest}) of $|{\bf c}(\lambda)|^{-2}$, we estimate the inner integral in the equation above as follows
\beas
&&\int_{\{\lambda\in \mathfrak a^\ast: |\lambda|<1\}} e^{-t|\lambda|^2}~(|\lambda|^2+|\rho|^2)^{\sigma+d}~d\lambda+ \int_{\{\lambda\in \mathfrak a^\ast: |\lambda|\geq 1\}} e^{-t|\lambda|^2}~(|\lambda|^2+|\rho|^2)^{\sigma+n-l}~d\lambda\\
&\leq& C_1+ C_2 \int_{1}^\infty e^{-t r^2}~r^{2(\sigma+n-l)}~r^{l-1}~dr\\
&\leq& C_1+C_2 ~ t^{-(\sigma+n-l/2)}.
\eeas
It now follows from (\ref{defn-I-p}) that
\beas
I_1\leq C\int_{0}^1 t^{-\frac{1}{2}(\sigma+n-l/2)}~e^{-\frac{|\rho|^2}{2}t}e^{-y^2/4t}~\frac{dt}{t^{1-\sigma}} < \infty.
\eeas
On the other hand 
\beas
I_2 &\leq& C \int_{1}^\infty \left(\int_{\mathfrak a^\ast}~e^{-1(|\lambda|^2+|\rho|^2)}~ (|\lambda|^2+|\rho|^2)^\sigma~|{\bf c}(\lambda)|^{-2}~d\lambda\right)^{\frac{1}{2}}~e^{-\frac{|\rho|^2}{2}t}~e^{-y^2/4t}~\frac{dt}{t^{1-\sigma}}\\
&\leq& C \|h_{1/2}\|_{H^\sigma(X)}.
\eeas
This completes the proof that $P_y^{-\sigma} \in H^\sigma(X)$.

The proofs of Hardy's inequalities is crucially depend on the following lemma.
\begin{lem} \label{lem-I}
 For $0<\sigma<1$ and $y>0$ we have,
$(-\Delta)^\sigma P_y^{-\sigma}(x)= \frac{4^\sigma\Gamma(\sigma)}{y^{2\sigma}\Gamma(-\sigma)}~P_y^\sigma(x)$.
\end{lem}
\begin{proof}
Let $f\in H^\sigma(X)$ and $u(x, y)=f \ast P_y^\sigma(x)$ be the solution of the extension problem (\ref{extn1}). For any $g\in L^2(X)$ we have by equation (\ref{expression-u-1}) that
\bea \label{eqn1}
\left\langle u(\cdot, y), g \right\rangle &=& \frac{1}{\Gamma(\sigma)}\int_{0}^\infty \left\langle e^{t\Delta}~(-\Delta)^{\sigma}f, g \right\rangle ~e^{-y^2/4t}~\frac{dt}{t^{1-\sigma}}\\
&=& \frac{1}{\Gamma(\sigma)}\int_{0}^\infty \int_{|\rho|^2}^\infty e^{-t \lambda}~\lambda^{\sigma} dE_{f, g}(\lambda)~e^{-y^2/4t}~\frac{dt}{t^{1-\sigma}} \nonumber\\
&=& \frac{1}{\Gamma(\sigma)}\int_{|\rho|^2}^\infty \left( \lambda^\sigma \int_{0}^\infty e^{-t\lambda}~t^{\sigma -1}~e^{-y^2/4t}~dt \right)~ dE_{f, g}(\lambda). \nonumber
\eea 
By using change of variable $t\ra y^2/(4\lambda r)$ we get the following formula (\cite{BRT}, p. 2582, equation(2.5)) 
\bes 
\lambda^\sigma \int_{0}^\infty e^{-t\lambda}~t^{\sigma -1}~e^{-y^2/4t}~dt
= \frac{y^{2\sigma}}{4^\sigma} \int_0^\infty e^{-t\lambda}~t^{-\sigma -1}~e^{-y^2/4t}~dt.
\ees
Using this in the equation above it follows that
\bea \label{eqn2}
\left\langle u(\cdot, y), g \right\rangle &=& \frac{y^{2\sigma}}{4^{\sigma} \Gamma(\sigma)} \int_0^\infty t^{-\sigma -1}~e^{-y^2/4t}~ \left(\int_{|\rho|^2}^\infty e^{-t\lambda}~ dE_{f,g}(\lambda)\right)~dt \nonumber\\
&=& \frac{y^{2\sigma}}{4^{\sigma} \Gamma(\sigma)} \int_0^\infty \left\langle e^{t\Delta}f, g \right\rangle t^{-\sigma -1}~e^{-y^2/4t}~ dt.
\eea
Therefore, from equations (\ref{eqn1}) and (\ref{eqn2}) we have
\bes
\int_{0}^\infty e^{t\Delta}~(-\Delta)^\sigma f(x)~e^{-y^2/4t}~t^{\sigma-1}dt= \frac{y^{2\sigma}}{4^{\sigma}} \int_{0}^\infty e^{t\Delta}f(x)~e^{-y^2/4t}~t^{-\sigma-1}~dt.
\ees
If we take the function  $f$ to be the heat kernel $h_{t_1}, t_1>0$, then the equation above reduces to 
\bes
\int_{0}^\infty (-\Delta)^\sigma(h_{t+t_1})(x)~e^{-y^2/4t}~t^{\sigma-1}dt= \frac{y^{2\sigma}}{4^{\sigma}} \int_{0}^\infty h_{t+t_1}(x)~e^{-y^2/4t}~t^{-\sigma-1}~dt.
\ees
Taking $t_1 \ra 0$, we get from dominated convergent theorem that
\be \label{relation}
\int_{0}^\infty (-\Delta)^\sigma h_t(x)~e^{-y^2/4t}~t^{\sigma-1}dt= \frac{y^{2\sigma}}{4^\sigma} \int_{0}^\infty h_t(x)~e^{-y^2/4t}~t^{-\sigma-1}~dt.
\ee
Using (\ref{relation}) and (\ref{defn-P}) we get
\beas
(-\Delta)^\sigma P_y^{-\sigma}(x)&=& \frac{y^{-2\sigma}}{4^{-\sigma} \Gamma(-\sigma)} \int_0^\infty (-\Delta)^\sigma h_t(x)~e^{-y^2/4t}~t^{\sigma-1}~dt\\
&=& \frac{1}{\Gamma(-\sigma)}\int_{0}^\infty h_t(x)~e^{-y^2/4t}~t^{-\sigma-1}~dt\\
&=& \frac{4^{\sigma} \Gamma(\sigma)}{y^{2\sigma} \Gamma(-\sigma)}~P_y^\sigma(x).
\eeas This completes the proof.
\end{proof}
We will now compute the asymptotic behaviour of the Poisson kernel $P_y^\sigma$ for arbitrary rank Riemannian symmetric spaces of noncompact type. We use this estimate crucially for the remaining part of this article. 
\begin{thm}\label{est-Pysigma} 
For $-1<\sigma<1, \sigma\not=0$ and $y>0$ we have
\beas
\Gamma(\sigma)\, P_y^\sigma (x)&\asymp & \frac{y^{2\sigma}}{4^\sigma} \sqrt{|x|^2+y^2}^{-l/2-1/2-\sigma-|\Sigma_0^+|} \phi_0(x) e^{-|\rho|\sqrt{|x|^2+y^2}}, \txt{ for } |x|^2+y^2 \geq 1,\\
&\asymp & y^{2\sigma}~\left(|x|^2+y^2\right)^{-n/2-\sigma} , \txt{ for } ~ |x|^2+y^2< 1.
\eeas	
\end{thm}
\begin{proof}
We first assume that $|x|^2+y^2< 1$. 
In this case, we will use the following local expansion of the heat kernel $h_t(x)$ 
\be \label{est-h-local}
h_t(x)=e^{-|x|^2/4t} t^{-n/2} v_0(x) + e^{-c|x|^2/t}  \mathcal O \left(t^{-n/2+1}\right) ,
\ee
where $v_0(x)=(4\pi)^{-n/2}+ \mathcal O (|x|^2)$ and $c< 1/4$ (\cite[(3.9), p. 278]{Anker}). Using this we have 
\beas
\Gamma(\sigma) P_y^\sigma(x)&=& \frac{y^{2\sigma}}{4^\sigma} \int_{0}^1 \left(e^{-|x|^2/4t} t^{-n/2} v_0(x) + e^{-c|x|^2/t}  \mathcal O \left(t^{-n/2+1}\right)\right) e^{-y^2/4t}~\frac{dt}{t^{1+\sigma}}\\
&& + \frac{y^{2\sigma}}{4^\sigma} \int_{1}^\infty h_t(x)~e^{-y^2/4t}~\frac{dt}{t^{1+\sigma}}\\
&=& \frac{y^{2\sigma}}{4^\sigma} v_0(x) \int_{0}^1 e^{-(|x|^2+y^2)/4t}~t^{-n/2-1-\sigma}~dt \\
&& + \frac{y^{2\sigma}}{4^\sigma} \int_{0}^1 e^{-(c|x|^2+y^2/4)/t}~\mathcal O(t^{-n/2+1})~t^{-1-\sigma}~dt + \frac{y^{2\sigma}}{4^\sigma} \int_{1}^\infty h_t(x)~e^{-y^2/4t}~\frac{dt}{t^{1+\sigma}}.
\eeas
We write the right-hand side of the equation above as $I_1 + I_2 + I_3$, where $I_1, I_2$ and $I_3$ are the first, second and third term respectively. Then applying  change of variable $s= \left(|x|^2+y^2\right)/(4t)$, we have 
\bes
I_1=4^{\frac n2} y^{2\sigma}~\left(|x|^2+y^2\right)^{-n/2-\sigma}~ v_0(x) \int_{(|x|^2+y^2)/4}^{\infty} e^{-s}~s^{n/2+ \sigma-1}~ds.
\ees
As $|x|^2 + y^2<1$, 
\bes
\int_{1}^\infty e^{-s}~s^{n/2+\sigma-1}~ds \leq \int_{(|x|^2+y^2)/4}^{\infty} e^{-s}~s^{n/2+\sigma-1}~ds \leq \int_{0}^\infty e^{-s}~s^{n/2+\sigma-1}~ds.
\ees
This implies that for $|x|^2+y^2< 1$
\bes
I_1\asymp y^{2\sigma} \left(|x|^2+y^2\right)^{-n/2-\sigma},
\ees
as $v_0(x)=(4\pi)^{-n/2}+ \mathcal O (|x|^2)$. For $I_2$, using $c< 1/4$ we have that
\beas
I_2 & \leq & C \frac{y^{2\sigma}}{4^\sigma} \int_{0}^1 e^{-c(|x|^2+y^2)/t}~\mathcal O(t^{-n/2+1})~t^{-1-\sigma}~dt\\
&\leq& C \,y^{2\sigma}~\left( |x|^2+y^2\right)^{-n/2-\sigma+1}~ \int_{c(|x|^2+y^2)}^\infty e^{-s}~s^{n/2+\sigma-2}~ds\\ 
&\leq & C\, y^{2\sigma}~\left( |x|^2+y^2\right)^{-n/2-\sigma}~\int_{0}^\infty  e^{-s}~s^{n/2+\sigma-1}~ds\\
&\leq & C\,  y^{2\sigma}~\left( |x|^2+y^2\right)^{-n/2-\sigma}.
\eeas
For the integral $I_3$,  using Theorem \ref{est-heatkernel-bothside} we get that for $|x|^2 + y^2<1$,
\bes
I_3= \frac{y^{2\sigma}}{4^\sigma} \int_{1}^\infty h_t(x)~e^{-y^2/4t}~\frac{dt}{t^{1+\sigma}}\leq C y^{2\sigma} \leq C y^{2\sigma} \left( |x|^2+y^2\right)^{-n/2-\sigma}.
\ees
This proves that for $|x|^2 + y^2<1$,
\bes
\Gamma(\sigma) P_y^\sigma(x) \asymp y^{2\sigma}~(|x|^2+y^2)^{-n/2-\sigma}.
\ees
We will now assume that $|x|^2+y^2\geq1$. Let us fix a positive number $\kappa >4$. We proceed as in the proof of \cite[Theorem 4.3.1]{AJ}.
\beas
\Gamma(\sigma) P_y^\sigma(x) &=& \frac{y^{2\sigma}}{4^\sigma}\int_{0}^{\infty} h_t(x)~e^{-y^2/4t}~\frac{dt}{t^{1+\sigma}}\\
&=& \frac{y^{2\sigma}}{4^\sigma}\{I_4+I_5+I_6\},
\eeas
where the quantities $I_4, I_5$ and $I_6$ are defined by the integration of the above integrand	$h_t(x)~e^{-y^2/4t}~ t^{-1-\sigma}$ over the intervals $\big[0, \kappa^{-1} b \big), \big[\kappa^{-1} b, \kappa b\big)$ and $\big[\kappa b, \infty\big)$ with $b=\sqrt{|x|^2+y^2}/(2|\rho|)$ respectively. For the integral $I_5$, using Theorem \ref{est-heatkernel-bothside} and the asymptotic of $\phi_0$ in Theorem \ref{thm-phi} (3) we get the following:
	\beas
	I_5 &\asymp& \int_{\frac{\kappa^{-1} \sqrt{|x|^2+y^2}}{2|\rho|}}^{\frac{\kappa \sqrt{|x|^2+y^2}}{2|\rho|}} t^{-n/2}(1+t)^{(n-l)/2- |\Sigma_0^+|} \left\{\prod_{\alpha\in \Sigma_0^+}(1+\alpha(x))\right\} e^{-|\rho|^2t-\rho(\log x)-|x|^2/4t} e^{-y^2/4t} \frac{dt}{t^{1+\sigma}}\\
	&\asymp& \int_{\frac{\kappa^{-1} \sqrt{|x|^2+y^2}}{2|\rho|}}^{\frac{\kappa \sqrt{|x|^2+y^2}}{2|\rho|}} t^{-l/2-|\Sigma_0^+|} \phi_0(x) e^{-|\rho|^2 t}e^{-(|x|^2+y^2)/4t} \frac{dt}{t^{1+\sigma}}\\
	&=&\int_{\kappa^{-1}}^\kappa (s\sqrt{|x|^2+y^2}/2|\rho|)^{-l/2-\sigma-1-|\Sigma_0^+|} \phi_0(x)e^{-s|\rho| \sqrt{|x|^2+y^2}/2}e^{-|\rho|\sqrt{|x|^2+y^2}/(2s)}\left(\frac{\sqrt{|x|^2+y^2}}{2|\rho|} \right)ds\\
	&\asymp& \left(\frac{\sqrt{|x|^2+y^2}}{2|\rho|}\right)^{-l/2-\sigma-|\Sigma_0^+|} \phi_0(x)\int_{\kappa^{-1}}^{\kappa}e^{-\sqrt{|x|^2+y^2}|\rho|(s+1/s)/2} ds.
	\eeas
The last both side estimate follows because 
\bes
\kappa^{-(l/2 + \sigma +1 + |\Sigma_0^+|)} \leq s^{-(l/2 + \sigma +1 + |\Sigma_0^+|)}\leq \kappa^{(l/2 + \sigma +1 + |\Sigma_0^+|)}.
\ees
Now, using the fact that 
\bes
\int_{\kappa^{-1}}^{\kappa} e^{-|\rho| \sqrt{|x|^2+y^2}(s+1/s)/2} ds\asymp |\rho|^{-1/2}(|x|^2+y^2)^{-1/4} e^{-|\rho| \sqrt{|x|^2+y^2}},
\ees
(this follows by the Laplace method \cite[Ch 5]{Cop})
we get from the above equation that 
\beas
I_5 &\asymp& \left(\sqrt{|x|^2+y^2}\right)^{-l/2-1/2-\sigma-|\Sigma_0^+|} \phi_0(x) e^{-|\rho| \sqrt{|x|^2+y^2}}.
\eeas
For the third integral $I_6$, we will use the fact that $\kappa> 4$. Using Theorem \ref{est-heatkernel-bothside}, we get
\beas
I_6 &\leq & \phi_0(x) \int_{\kappa \sqrt{|x|^2+y^2}/(2|\rho|)}^\infty t^{-l/2-|\Sigma_0^+|-1-\sigma}e^{-|\rho|^2t} e^{-(|x|^2+y^2)/4t}dt\\
&\leq& \phi_0(x) \left(\kappa \sqrt{|x|^2+y^2}/(2|\rho|)\right)^{-l/2-|\Sigma_0^+|-1}\int_{\kappa\sqrt{|x^2|+y^2}/(2|\rho|)}^{\infty} t^{-\sigma}e^{-|\rho|^2t} e^{- (|x|^2+y^2)/(4t)}dt\\
&\leq& C \phi_0(x) \left(\sqrt{|x|^2+y^2}\right)^{-l/2-|\Sigma_0^+|-1} ~e^{-|\rho|^2 k \sqrt{|x|^2+y^2}/(4|\rho|)} \int_{\kappa\sqrt{|x^2|+y^2}/(2|\rho|)}^{\infty} t^{-\sigma} e^{-|\rho|^2t/2} e^{- (|x|^2+y^2)/(4t)}dt\\
&\leq& C \left(\sqrt{|x|^2+y^2}\right)^{-l/2-|\Sigma_0^+|-1/2}~\phi_0(x)e^{-(|\rho|+\eta)\sqrt{|x|^2+y^2}},
\eeas
where $\eta= |\rho| \kappa/4 -|\rho|>0$.
For the first integral $I_4$, we use heat kernel Gaussian estimate (\ref{est-ht}) and the estimate of $\phi_0$ in Theorem \ref{thm-phi} to obtain the following
\beas
I_4 &\leq&  \int_{0}^{\kappa^{-1}\sqrt{|x|^2+y^2}/(2|\rho|)} t^{-d_1}(1+|x|)^{d_2}e^{-|\rho|^2t -\rho(\log x)} e^{-(|x|^2+y^2)/(4t)} \frac{dt}{t^{1+\sigma}}\\
&\leq& (1+|x|)^{d_2-|\Sigma_0^{+}|}\phi_0(x) \int_{0}^{\kappa^{-1}\sqrt{|x|^2+y^2}/(2|\rho|)}e^{-(|x|^2+y^2)/(4t)}t^{-1-\sigma-d_1}dt\\
&=& (1+|x|)^{d_2-|\Sigma_0^{+}|}\phi_0(x) \int_{0}^{\kappa^{-1}\sqrt{|x|^2+y^2}/(2|\rho|)}e^{-(|x|^2+y^2)/(8t)} e^{-(|x|^2+y^2)/(8t)} t^{-1-\sigma-d_1}dt\\
&\leq& C(1+|x|)^{d_2-|\Sigma_0^{+}|}\phi_0(x) \,e^{-|\rho|\kappa \frac{\sqrt{x^2 + y^2}}{4}}~\int_{0}^{\kappa^{-1}\sqrt{|x|^2+y^2}/(2|\rho|)}e^{-(|x|^2+y^2)/(8t)} t^{-1-\sigma-d_1}dt\\ 
&\leq & C(1+|x|)^{d_2-|\Sigma_0^{+}|}\phi_0(x) \,e^{-(|\rho|+\epsilon)\sqrt{|x|^2 + y^2}}  (|x|^2 + y^2)^{-\sigma-d_1},
\eeas
for some $\epsilon>0$, as $\kappa>4$.

This completes the proof.
\end{proof}

To prove Hardy's inequalities we use an integral representation for the operator $(-\Delta)^{\sigma}$. The following function 
\be \label{defn-p0}
P_0^\sigma(x)= \int_{0}^\infty h_t(x) \frac{dt}{t^{1+\sigma}},
\ee
serves as the kernel of the integral representation. We state both sides estimate of $P_0^\sigma$, whose proof is exactly the same as of Theorem \ref{est-Pysigma}.
\begin{thm}\label{est-Pysigma-0} For any $\alpha > -n/2 $ the following asymptotic estimates holds:
\beas
P_0^\alpha (x)&\asymp &  |x|^{-l/2-1/2-\alpha-|\Sigma_0^+|} \phi_0(x) e^{-|\rho| |x|}, \txt{ for } |x| \geq 1,\\
&\asymp & |x|^{-n-2\alpha} , \txt{ for } ~ |x|< 1.
\eeas	
\end{thm}
\begin{cor} \label{cor-p0}
Let $\chi$ be the characteristic function of the unit ball in $X$ and $\alpha>0$. Then the function $(1-\chi)P_0^\alpha$ is in $L^p(X)$ for $1\leq p\leq \infty$.
\end{cor}
\begin{proof}
For $1< p\leq \infty$, the result follows trivially from the asymptotic formula in Theorem \ref{est-Pysigma-0}. We prove the case $p=1$. We recall from  (\ref{integral-X}) that
\bes
\int_{\{x\in X:|x|>1\}} P_0^\alpha(x)~dx \leq C\int_{\{H\in \overline{\mathfrak a_+}: |H|>1\}} P_0^\alpha(\exp H)~e^{2\rho(H)}~dH.
\ees
Let $\Gamma$ be a small circular cone in $\overline{\mathfrak a_+}$ around the $\rho$-axis. By introducing polar coordinates in $\Gamma$ and using (\ref{estimate-phi0}) we get
\beas
&&\int_{\{H\in \Gamma: |
H|>1\}} P_0^\alpha(\exp H)~e^{2\rho(H)}~dH\\
&\leq& C\int_{\{H\in \Gamma: |H|>1\}} |H|^{-l/2-1/2-\alpha}~e^{\rho(H)-|\rho||H|}~dH \\
&\leq& C\int_{1}^\infty r^{-l/2-1/2-\alpha} r^{l-1}~\int_{0}^\nu (\sin \xi)^{l-2}~e^{-r(1-\cos \xi)}~d\xi~dr.
\eeas
Since $\sin \xi \sim \xi$ and $1-\cos \xi \sim \xi^2$, the inner integral behaves like $r^{1/2-l/2}$. Consequently, the integral above is finite. On the other hand, $e^{\rho(H)-|\rho||H|}$ decreases exponentially outside $\Gamma$, and therefore
\bes
\int_{\{H\in \overline{\mathfrak a_+}\backslash \Gamma: |
H|>1\}} P_0^\alpha(\exp H)~e^{2\rho(H)}~dH=\int_{\{H\in \overline{\mathfrak a_+}\backslash\Gamma:|H|>1\}} H^{-l/2-1/2-\alpha}~e^{\rho(H)-|\rho||H|}~dH< \infty.
\ees
This completes the proof.
\end{proof}

\section{Fractional Hardy inequalities}
This  section aims to prove two versions of the Hardy's inequalities for fractional powers of the Laplace-Beltrami operator on $X$, namely Theorem \ref{Hardy-inhomogeneous} and Theorem \ref{Hardy-homogeneous} with homogeneous and non-homogeneous weight functions respectively.
In order to prove these inequalities, we will follow similar ideas used by Frank et al. \cite{FLS} in the case of Euclidean Laplacian. Therefore, we need to
establish ground state representations for the operators $(-\Delta)^\sigma$. We start with the following integral representations of $(-\Delta)^\sigma$ on $X$. For the cases of real hyperbolic spaces, analogues integral representations were proved in \cite[Theorem 2.5]{BGS}.

\begin{lem} \label{lem-int}
Let $0 < \sigma< 1/2$. Then for all $f\in C_c^\infty(X)$ we have
\bes
(-\Delta)^\sigma f(x)=\frac{1}{|\Gamma(-\sigma)|}~\int_X \left(f(x)-f(z)\right)P_0^\sigma(z^{-1}x)~dz,
\ees
where $P_0^\sigma$ is defined in (\ref{defn-p0}).
\end{lem}
\begin{proof}
Let $f\in C_c^\infty(X)$. Using the numerical identity
\bes
\lambda^\sigma =\frac{1}{|\Gamma(-\sigma)|} \int_{0}^\infty \left(1-e^{-t \lambda}\right)~\frac{dt}{t^{1+\sigma}}, \:\: \lambda>0,
\ees 
and the spectral theorem we have
\bes
(-\Delta)^\sigma f(x)=\frac{1}{|\Gamma(-\sigma)|} \int_{0}^\infty \left(f(x)-e^{t\Delta}f(x)\right)~\frac{dt}{t^{1+\sigma}}.
\ees
By (\ref{prop-ht}) it follows that
\be \label{u-f}
f(x)-e^{t\Delta} f(x)= f(x)-f\ast h_t(x)=\int_{X}(f(x)-f(xz^{-1}))~h_t(z)~dz.
\ee
Thus, we have the following representation
\bes
(-\Delta)^\sigma f(x) = \frac{1}{|\Gamma(-\sigma)|} \int_{0}^\infty \int_X \left(f(x)-f(xz^{-1})\right) h_t(z)~dz~\frac{dt}{t^{1+\sigma}}.
\ees
We now show that the right-hand side is absolutely integrable and hence, interchange of the order of integral is possible. Then the result follows by the change of variable $z\mapsto z^{-1}x$. To show absolute integrability let us define 
\beas
&& I_1=\frac{1}{|\Gamma(-\sigma)|} \int_{0}^\infty \int_{\{z\in X: |z|< 1\}} \left|f(x)-f(xz^{-1})\right| h_t(z)~dz~\frac{dt}{t^{1+\sigma}},\\
&& I_2=\frac{1}{|\Gamma(-\sigma)|} \int_{0}^\infty \int_{\{z\in X:|z|\geq 1\}}  \left|f(x)-f(xz^{-1})\right| h_t(z)~dz~\frac{dt}{t^{1+\sigma}}.
\eeas
For the integral $I_2$, we use the fact that $P_0^\sigma\in L^1(X)$ away from the origin (Corollary \ref{cor-p0}). Indeed, we have that 
\bes 
\int_{\{z\in X:|z|\geq 1\}} \int_{0}^\infty |f(x)-f(xz^{-1})|~h_t(z)~\frac{dt}{t^{1+\sigma}}~dz\leq  \|f\|_{L^\infty(X)}\int_{\{z\in X: |z|\geq 1\}} P_0^\sigma(z)~dz< \infty. 
\ees
Therefore, by Fubini's theorem $I_2$ is also finite. For $I_1$ we first observe by the fundamental theorem of calculus (see the proof of equation (34) in \cite{A}) that
\be \label{mvt}
|f(x)-f\left(x(\exp H)\right)|\leq |H|\int_{0}^1 |\nabla f\left(x \exp(sH)\right)|~ds \leq |H|~\|\nabla f\|_{L^\infty(X)},
\ee 
for $x\in X, H\in \mathfrak a$. Using the above estimate and the fact that $P_0^\sigma(x)\asymp |x|^{-n-2\sigma}$ around the origin (Theorem \ref{est-Pysigma-0}) it follows that
\beas
&& \int_{|z|< 1}\int_{0}^\infty |f(x)-f(xz^{-1})|~ h_t(z)~\frac{dt}{t^{1+\sigma}}~dz\\
&\leq& C \|\nabla f\|_{L^\infty(X)}~\int_{\{H\in\overline{\mathfrak a_+}: |H|<1\}} |H|~|H|^{-n-2\sigma}~J(\exp H)~dH\\
&\leq & C \|\nabla f\|_{L^\infty(X)} \int_{0}^1 r^{1-n-2\sigma}~r^{n-1}~dr,
\eeas
and the right-hand side is finite if $0< \sigma<1/2$. This completes the proof. 
\end{proof}

\begin{rem}
If $rank(X)=1$, then for $1/2\leq \sigma<1$ the integral formula in Lemma \ref{lem-int} exists in principal value sense. To see this, let $\mathfrak a= span\{H_0\}$ with $|H_0|=1$. Clearly, for $\sigma>0$, the integral $I_2$ is absolutely convergent and we can interchange the order of the integral. On the other hand the formula (\ref{integral-X}) yields
\bes
I_1=\frac{1}{|\Gamma(-\sigma)|} \int_{0}^\infty \int_{-1}^1 \big(f(x)-f\left(x \exp(-sH_0)\right)\big)~ h_t\left(\exp (sH_0)\right)~J\left(\exp(sH_0)\right)~ds\frac{dt}{t^{1+\sigma}}.
\ees
We now define $F(s):= f\left(x \exp(sH_0)\right)$, for $s\in \R$. Since $f\in C_c^\infty(X)$, it follows that for each $x\in X$, the function $F\in C_c^\infty(\R)$. By using the Taylor development of $F$, we get that
\bes
I_1= \frac{1}{|\Gamma(-\sigma)|} \int_{0}^\infty \int_{-1}^1 \left(s F'(x)+ \frac{s^2 F''(x)}{2!}+ \mathcal O(s^3)\right)~ h_t\left(\exp (sH_0)\right)~J\left(\exp(sH_0)\right)~ds\frac{dt}{t^{1+\sigma}}.
\ees 
Since the heat kernel $h_t$ and the Jacobian $J$ is even, the first order term vanishes. Hence, using the fact that $P_0^\sigma(x)\sim |x|^{-n-2\sigma}$, around the origin (Theorem \ref{est-Pysigma-0}), it follows that 
\beas
I_1 &\leq& C_f \int_{0}^\infty \int_{0}^1 s^2~ h_t\left(\exp(sH_0)\right)~s^{n-1}~ds~\frac{dt}{t^{1+\sigma}}= C_f \int_{0}^1 ~s^{n+1}~s^{-n-2\sigma}~ds,
\eeas
which is finite if $0< \sigma< 1$. Hence, the required integral formula exists as a principal value sense. For the case of higher rank symmetric spaces, neither the heat kernel $h_t\left(\exp(\cdot)\right)$ nor the Jacobian $J\left(\exp(\cdot)\right)$ is, in general, radial function on $\mathfrak a$. They are only Weyl group invariant. This is the main difficulty that we could not prove the integral formula in the lemma above for $1/2\leq \sigma<1$ in case of $rank(X)>1$.
\end{rem}

\begin{lem} \label{lem-int2}
Let $0 < \sigma < 1$. Then, for all $f \in H^\sigma(X)$
\bes
\langle (-\Delta)^\sigma f, f\rangle = \frac{1}{2|\Gamma(-\sigma)|}\int_X \int_X \left|f(z)-f(x)\right|^2P_0^\sigma(z^{-1}x)~dz ~dx.
\ees
\end{lem}
\begin{proof}
We first prove that for $0< \sigma< 1$ and $f \in C_c^\infty(X)$ the quantity 
\be \label{int-formula}
\frac{1}{2|\Gamma(-\sigma)|}\int_X \int_X \left|f(z)-f(x)\right|^2P_0^\sigma(z^{-1}x)~dz ~dx< \infty.
\ee
To show this let us assume $supp~ f\subset {\bf B}(o, m)$ for some $m>1$ and define
\beas
&& I_1= \frac{1}{2|\Gamma(-\sigma)|}\int_{{\bf B}(o, 2m)} \int_X \left|f(z)-f(x)\right|^2P_0^\sigma(z^{-1}x)~dz ~dx,\\
&& I_2 = \frac{1}{2|\Gamma(-\sigma)|} \int_{X\backslash {\bf B}(o, 2m)} \int_{X} \left|f(z)-f(x)\right|^2~P_0^\sigma(z^{-1}x)~dz ~dx.
\eeas
Since $supp~f \subset {\bf B}(o, m)$ it follows that
\beas
I_2 &=& \frac{1}{2|\Gamma(-\sigma)|} \int_{X\backslash {\bf B}(o, 2m)} \int_{{\bf B}(o, m)} \left|f(z)\right|^2~P_0^\sigma(z^{-1}x)~dz ~dx\\
&\leq& \frac{\|f\|_{L^\infty(X)}^2}{2|\Gamma(-\sigma)|} \int_{{\bf B}(o, m)} \int_{X \backslash {\bf B}(o, 2m)} P_0^\sigma(z^{-1}x)~dx~dz\\
&\leq& \frac{\|f\|_{L^\infty(X)}^2}{2|\Gamma(-\sigma)|} |{\bf B}(o, m)| \int_{X \backslash {\bf B}(o, m)} P_0^\sigma(x)~dx< \infty.
\eeas
The last term is finite because of the fact that $P_0^\sigma$ is integrable away from the origin (Corollary \ref{cor-p0}). To show that $I_1$ is finite we write it as follows
\beas
I_1 &=&\int_{{\bf B}(o, 2m)} \int_{{\bf B}(0, 3m)} \left|f(z)-f(x)\right|^2P_0^\sigma(z^{-1}x)~dz ~dx\\
&& + \int_{{\bf B}(o, 2m)} \int_{X\backslash{\bf B}(0, 3m)} \left|f(z)-f(x)\right|^2P_0^\sigma(z^{-1}x)~dz ~dx.
\eeas
Using change of variable $z\mapsto xz^{-1}$ in the first integral, the estimate (\ref{mvt}) and the asymptotic estimates of $P_0^\sigma$ in Theorem \ref{est-Pysigma-0} it follows that
\beas
I_1 &\leq&\int_{{\bf B}(o, 2m)} \int_{{\bf B}(o, 5m)} \left|f(xz^{-1})-f(x)\right|^2P_0^\sigma(z)~dz ~dx\\
&& + \int_{{\bf B}(o, 2m)} \int_{X\backslash{\bf B}(o, 3m)} \left|f(z)-f(x)\right|^2P_0^\sigma(z^{-1}x)~dz ~dx \\
&\leq & C \|\nabla f\|_{L^\infty(X)}^2 \int_{{\bf B}(o, 2m)}~dx~\int_{\{H\in \overline{\mathfrak a_+}: |H|<5m\}} |H|^2 ~H^{-n-2\sigma}~J(\exp H)~dH\\
&& + C \int_{{\bf B}(o, 2m)}   \|f\|_{L^\infty(X)}^2~\int_{X\backslash {\bf B}(o, 3m)} P_0^\sigma(z^{-1} x)~dz ~dx\\
&\leq& C\|\nabla f\|_{L^\infty(X)}^2 \int_{0}^{5m} r^{2-n-2\sigma}~r^{n-1}~dr + C\|f\|_{L^\infty(X)}^2  \int_{X\backslash {\bf B}(o, m)} P_0^\sigma(z)~dz.
\eeas
The first term of the above quantity is finite provided $\sigma< 1$ and the second one finite by Corollary \ref{cor-p0}. This completes the proof of the fact the quantity in (\ref{int-formula}) is finite. 

Let $0< \sigma< 1/2$ and $f\in C_c^\infty(X)$. By the integral representation in Lemma \ref{lem-int} it follows that
\bes
\left\langle (-\Delta)^\sigma f, f\right\rangle = \frac{1}{|\Gamma(-\sigma)|}\int_X \int_X \left(f(x)-f(z)\right)P_0^\sigma(z^{-1}x) ~\overline{f(x)}~dz ~dx.
\ees
As the kernel $P_0^\sigma$ is symmetric, that is $P_0^\sigma(x)= P_0^\sigma(x^{-1})$, the above quantity is also equals to 
\bes
\frac{1}{|\Gamma(-\sigma)|}\int_X \int_X \left(f(z)-f(x)\right)P_0^\sigma(z^{-1}x) ~\overline{f(z)}~dx ~dz.
\ees
By adding them up we get that
\bes 
\langle (-\Delta)^\sigma f, f\rangle = \frac{1}{2|\Gamma(-\sigma)|}\int_X \int_X \left|f(z)-f(x)\right|^2P_0^\sigma(z^{-1} x)~dz ~dx.
\ees
The justification of the change of order of integration follows from (\ref{int-formula}). By the analytic continuation, we extend the range of $\sigma$ to $0< \sigma<1$ provided $f\in C_c^\infty(X)$. Indeed, the functions $\sigma \mapsto -\Gamma(-\sigma)$  and $\sigma \mapsto \langle (-\Delta)^\sigma f, f \rangle$ are holomorphic on $S=\{w\in \C: 0<\Re w <1\}$. Hence their product $F(\sigma)= -\Gamma(-\sigma)~ \langle (-\Delta)^\sigma f, f \rangle$ is also holomorphic on $S$. On the other hand, since right-hand side of (\ref{int-formula}) is finite for $0<\sigma<1$, by the Morera's theorem it follows that the function $G$ defined by
\bes
G(\sigma)=\frac{1}{2} \int_X \int_X \left|f(z)-f(x)\right|^2P_0^\sigma(z^{-1}x)~dz ~dx.
\ees
is holomorphic on $S$. Since $F(\sigma)=G(\sigma)$ for $0< \sigma<1/2$ we get that $F(\sigma)=G(\sigma)$ for all $\sigma\in S$, in particular, for $0< \sigma<1$.

By approximating any function $f\in H^\sigma(X)$ by a sequence of functions $f_k\in C_c^\infty(X)$, we complete the proof. This uses the fact that $P_0^\sigma(x)\asymp |x|^{-n-2\sigma}$ around the origin and the rest follows as in the proof of Lemma 5.1 in \cite{RT2}. 
\end{proof}
We now establish ground state representation for the operator $(-\Delta)^\sigma$ as a consequence of the integral representation proved in Lemma \ref{lem-int2}. As in the Euclidean case, we define the following error term. For $0<\sigma<1$ and $y>0$ we let,
\bes
H_y^\sigma[F]= \langle (-\Delta)^\sigma F, F\rangle -\frac{4^\sigma \Gamma(\sigma)}{y^{2\sigma} \Gamma(-\sigma)} \int_{X} |F(x)|^2 \left(\frac{P_y^\sigma(x)}{P_y^{-\sigma}(x)}\right)\,dx.
\ees
\begin{thm} \label{thmgsr}
Let $0<\sigma<1$ and $y>0$. If $F\in C_c^{\infty}(X)$ and $G(x)= F(x)~\left(P^{-\sigma}_y(x)\right)^{-1}$ then
\bes
H_y^\sigma[F]= \frac{1}{2|\Gamma(-\sigma)|}  \int_X\int_X |G(x)-G(z)|^2~P_y^{-\sigma}(x) P_y^{-\sigma}(z)~P_0^{-\sigma}(z^{-1}x)~dx~dz.
\ees
\end{thm}

\begin{proof} 
Let $f, g\in H^\sigma(X)$. From Lemma \ref{lem-int2} we get that
\be \label{reptn1}
\left \langle (- \Delta)^\sigma f, g\right\rangle = \frac{1}{2|\Gamma(-\sigma)|}\int_X \int_X \left(f(z)-f(x)\right)~\overline{\left(g(z)-g(x)\right)}~P_0^\sigma(z^{-1}x)~dz ~dx.
\ee
Let us assume $g=P_y^{-\sigma}$, and $f(x)=|F(x)|^2~g(x)^{-1}$. Then the right-hand side of (\ref{reptn1}) reduces to
\bea \label{reptn4}
&& \frac{1}{2|\Gamma(-\sigma)|} \int_{X}\int_X \left(\frac{|F(z)|^2}{g(z)}-\frac{|F(x)|^2}{g(x)}\right)~\overline{\left(g(z)- g(x)\right)}~P_0^\sigma(z^{-1}x)~dz ~dx\\
&=& \frac{1}{2|\Gamma(-\sigma)|} \int_X\int_X \left(|F(x)-F(z)|^2- \left|\frac{F(x)}{g(x)}-\frac{F(z)}{g(z)}\right|^{2} ~ g(x)g(z)\right)~ P_0^\sigma(z^{-1}x)~dz ~dx \nonumber.
\eea
Also, using Lemma \ref{lem-I} the left-hand side of (\ref{reptn1}) reduces to
\beas
\left\langle (-\Delta)^\sigma f, g\right\rangle &=& \left\langle (-\Delta)^\sigma(|F(x)|^2/g(x)), g(x) \right\rangle\\
&=& \left\langle \left(|F(x)|^2/g(x)\right), (-\Delta)^\sigma P_y^{-\sigma}\right\rangle\\
&=& \frac{4^\sigma \Gamma(\sigma)}{y^{2\sigma} \Gamma(-\sigma)}  \left\langle (|F(x)|^2/g(x)), P_y^{\sigma}\right\rangle\\
&=&\frac{4^{\sigma} \Gamma(\sigma)}{y^{2\sigma} \Gamma(-\sigma)} \int_{X} |F(x)|^2 \frac{ P_y^\sigma(x)}{P_y^{-\sigma}(x)}~dx.
\eeas
Therefore, equating the left-hand and right-hand sides of the equation (\ref{reptn1}) we have
\beas
&& \frac{4^{\sigma} \Gamma(\sigma)}{y^{2\sigma} \Gamma(-\sigma)} \int_{X} |F(x)|^2 \frac{ P_y^\sigma(x)}{P_y^{-\sigma}(x)}~dx = \frac{1}{2|\Gamma(-\sigma)|}\int_X \int_X \left|F(x)-F(z)\right|^2P_0^\sigma(z^{-1}x)~dx ~dz \\
&-& \frac{1}{2|\Gamma(-\sigma)|} \int_X\int_X \left|\frac{F(x)}{g(x)}-\frac{F(z)}{g(z)}\right|^{2} ~ g(x)g(z)~ P_0^\sigma(z^{-1}x)~dz ~dx. 
\eeas
By Lemma \ref{lem-int2} the first term in the right-hand side of the above equation is equals to $\langle (-\Delta)^\sigma F, F \rangle$.  Hence, it follows that
\beas
&& \langle (-\Delta)^\sigma F, F \rangle - \frac{4^\sigma \Gamma(\sigma)}{y^{2\sigma} \Gamma(-\sigma)}  \int_X |F(x)|^2 \frac{P_y^\sigma(x)}{P_y^{-\sigma}(x)}~dx\\
&=& \frac{1}{2|\Gamma(-\sigma)|}  \int_X\int_X |G(x)-G(z)|^2~P_y^{-\sigma}(x) P_y^{-\sigma}(x)~P_0^{-\sigma}(z^{-1}x)~dx~dz,  
\eeas
where $G(x)= F(x)P_y^{-\sigma}(x)^{-1}$. This completes the proof.
\end{proof}

We have already observed that for $0 < \sigma < 1$,  $\Gamma(-\sigma)<0$ and hence $P_y^{-\sigma}\leq 0$.   Therefore, as a corollary of Theorem \ref{thmgsr}  we get the following result.
\begin{cor} \label{Hardy-inhomogeneous-1}
For a fixed $y>0$ and $0<\sigma<1$ we have
\bes
\left\langle (-\Delta)^\sigma F, F\right\rangle \geq  \frac{4^\sigma}{y^{2\sigma}} \int_{X} |F(x)|^2 \left(\frac{\Gamma(\sigma) }{\Gamma(-\sigma)}\frac{P_y^\sigma(x)}{P_y^{-\sigma}(x)}\right)\,dx, \:\: \textit{ for } F\in H^\sigma(X).
\ees
\end{cor}
\begin{rem}
By Lemma \ref{lem-I} it follows that the equality in the expression above is achieved for the function $F=P_y^{-\sigma}$. Therefore, the constant $4^\sigma \Gamma(\sigma)/ y^{2\sigma}|\Gamma(-\sigma)|$ appeared in the corollary above is sharp.  
\end{rem}

Now, using the estimate of $P_y^\sigma$ (Theorem \ref{est-Pysigma}) in Corollary \ref{Hardy-inhomogeneous-1} we get Theorem \ref{Hardy-inhomogeneous}.
\begin{proof}[Proof of Theorem \ref{Hardy-inhomogeneous}]
From Theorem \ref{est-Pysigma} we have 
\begin{equation}\nonumber
\frac{\Gamma(\sigma) }{\Gamma(-\sigma)}\frac{P_y^\sigma(x)}{P_y^{-\sigma}(x)} \asymp \left\{ \begin{array}{lll}
\frac{y^{4\sigma}}{(|x|^2 + y^2)^\sigma} & \text{ if } |x|^2 + y^2\geq 1 \\

\frac{y^{4\sigma}}{(|x|^2 + y^2)^{2\sigma}} & \text{ if } |x|^2 + y^2< 1 .
\end{array}    \right.
\end{equation}
Therefore, from Corollary \ref{Hardy-inhomogeneous-1} we have 
\begin{equation}
\nonumber
\left\langle (-\Delta)^\sigma F, F\right\rangle \geq   C_\sigma\, y^{2\sigma} \left( \int_{\left\{x:|x|^2+y^2<1\right\}} \frac{|F(x)|^2}{(y^2 +|x|^2)^{2\sigma}}\,dx + \int_{\left\{x:|x|^2+y^2\geq 1\right\}} \frac{|F(x)|^2}{(y^2 +|x|^2)^{\sigma}}\,dx\right).
\end{equation}

\end{proof}

We now prove Hardy's inequality corresponding to the homogeneous weight function (Theorem \ref{Hardy-homogeneous}). To prove this theorem  we need the following expression of the error term. 
\begin{thm}\label{ground-state-0}
Let $0<\sigma<1$ and $\alpha> (2\sigma+n)/4$. Then for  $F\in C_c^{\infty}(X)$ and $G(x)= F(x)~\left(P_0^{-\alpha}(x)\right)^{-1}$ we have 
\bea
&& \left\langle (-\Delta)^\sigma F, F\right\rangle -\frac{\Gamma(\alpha)}{\Gamma(\alpha -\sigma)} \int_{X} |F(x)|^2 \left(\frac{P_0^{\sigma-\alpha}(x)}{P_0^{-\alpha}(x)}\right)\,dx \nonumber \\&=& \frac{1}{2|\Gamma(-\sigma)|}  \int_X\int_X |G(x)-G(z)|^2~P_0^{-\alpha}(x) P_0^{-\alpha}(z)~P_0^{\sigma}(z^{-1}x)~dx~dz,
\eea
where the function $P_0^{-\alpha}$ is defined by (\ref{defn-p0}).
\end{thm}
\begin{proof}
Since $\alpha>n/4$, we observe from Theorem \ref{est-Pysigma-0} that $P_0^{-\alpha} \in L^2(X)$.  As before by Fubini theorem the spherical Fourier transform of $P_0^{-\alpha}$ is given by 
\bes
\widehat{P_0^{-\alpha}}(\lambda) = \int_0^\infty e^{-t(|\lambda|^2 + |\rho|^2)} \frac{dt}{t^{1-\alpha}} = \Gamma(\alpha)~\left(|\lambda|^2 + |\rho|^2\right)^{-\alpha}, \:\: \lambda\in \mathfrak a^\ast.
\ees
Since $\alpha> (2\sigma+n)/4$, it follows that $P_0^{-\alpha}\in H^\sigma(X)$. Indeed, using (\ref{clambdaest}) we get that
\beas
\int_{\mathfrak a^\ast}|\widehat{P_0^{-\alpha}}(\lambda)|^2~(|\lambda|^2+|\rho|^2)^\sigma~|{\bf c}(\lambda)|^{-2}~d\lambda\leq C+ C'\int_{\{\mathfrak a^\ast: |\lambda|\geq 1\}} (|\lambda|^2+|\rho|^2)^{-2\alpha+\sigma}~(1+|\lambda|)^{n-l} ~d\lambda,
\eeas
which is finite. We recall from (\ref{reptn1}) that for $f, g \in H^\sigma(X)$
\begin{equation} \label{reptn3}
\left \langle (- \Delta)^\sigma f, g\right\rangle = \frac{1}{2|\Gamma(-\sigma)|}\int_X \int_X \left(f(z)-f(x)\right)~\overline{\left(g(z)-g(x)\right)}~P_0^\sigma(z^{-1}x)~dz ~dx.
\end{equation}
If we put $g(x)=P_0^{-\alpha}(x)$ and $f(x)=|F(x)|^2 (P_0^{-\alpha}(x))^{-1}$ in the equation above, then the left-hand side reduces to
\beas
\left\langle (-\Delta)^\sigma f, g \right\rangle &=& \int_{\mathfrak{a}^\ast} \left(|\lambda|^2 + |\rho|^2\right)^\sigma \widehat{f}(\lambda) \,\widehat{g}(\lambda) \,|{\bf c}(\lambda)|^{-2}\,d\lambda \\
&=& \Gamma(\alpha) \int_{\mathfrak{a}^\ast} \left(|\lambda|^2 + |\rho|^2\right)^{\sigma-\alpha} \widehat{f}(\lambda) \, |{\bf c}(\lambda)|^{-2}\,d\lambda \\ 
&=&\frac{\Gamma(\alpha)}{\Gamma(\alpha-\sigma)} \int_{\mathfrak{a}^\ast} \widehat{P_0^{\sigma-\alpha}}(\lambda) ~\widehat{f}(\lambda) \, |{\bf c}(\lambda)|^{-2}\,d\lambda \\ 
&=& \frac{\Gamma(\alpha)}{\Gamma(\alpha-\sigma)} \int_X |F(x)|^2 \frac{P_0^{\sigma-\alpha}(x)}{P_0^{-\alpha}(x)} \,dx.
\eeas
The right-hand side of the equation (\ref{reptn3}) becomes (see (\ref{reptn4}))
\begin{equation}
\frac{1}{2|\Gamma(-\sigma)|} \int_X\int_X \left(|F(x)-F(z)|^2- \left|\frac{F(x)}{g(x)}-\frac{F(z)}{g(z)}\right|^{2} ~ g(x)g(z)\right)~ P_0^\sigma(z^{-1}x)~dz ~dx.
\end{equation}
Hence, equating both sides of the equation (\ref{reptn3}) we have 
\beas
&& \frac{\Gamma(\alpha)}{\Gamma(\alpha-\sigma)}\int_X |F(x)|^2 \frac{P_0^{\sigma-\alpha}(x)}{P_0^{-\alpha}(x)} \,dx =\frac{1}{2|\Gamma(-\sigma)|} \int_X\int_X |F(x)-F(z)|^2~ P_0^\sigma(z^{-1}x)~dz ~dx \\
&& - \frac{1}{2|\Gamma(-\sigma)|} \int_X\int_X  \left|\frac{F(x)}{g(x)}-\frac{F(z)}{g(z)}\right|^{2} ~ g(x)g(z) ~ P_0^\sigma(z^{-1}x)~dz ~dx.
\eeas
By Lemma \ref{lem-int2} the first term in the right-hand side of the above equation is equals to $\langle (-\Delta)^\sigma F, F \rangle$ and hence the required identity follows.
\end{proof}

\begin{proof}[Proof of Theorem \ref{Hardy-homogeneous}]
Since $\sigma <1$ and $n \geq 2$, we can choose a positive $\alpha$ such that $2\sigma+n/4< \alpha< n/2$.
 From Theorem \ref{est-Pysigma-0} above it follows that
\beas
\frac{P_0^{\sigma-\alpha}(x)}{P_0^{-\alpha}(x)} &\asymp& |x|^{-2\sigma}, \:\:
\textit{ for } |x|< 1; \\ 
&\asymp & |x|^{-\sigma}, \:\: \textit{
for } |x|\geq 1.
\eeas
Therefore, it follows from Theorem \ref{ground-state-0} that
\begin{equation}\nonumber
\left\langle (-\Delta)^\sigma F, F \right\rangle \geq C_\sigma \left(\int_{|x|< 1} \frac{|F(x)|^2}{|x|^{2\sigma}} \,dx+ \int_{|x|\geq 1} \frac{|F(x)|^2}{|x|^{\sigma}}\,dx\right).
\end{equation} 

\end{proof}

\section{Mapping properties of Poisson Operator}
In this section we prove Theorem \ref{mapping-property}. We start with the following lemma.
\begin{lem}\label{lem-Py-sigma-lq}
For $0<\sigma<1$ and $1 < q< \frac{n+1}{n}$, the function $(x, y)\mapsto P_y^\sigma(x) \in L^q(X \times \R_+)$.
\end{lem}
\begin{proof} 
We first observe from (\ref{J-est}) that for $H\in \mathfrak a$ with $|H|<1$, the Jacobian $J(\exp H)$ corresponding to the polar decomposition is of order $|H|^{n-l}$. From Theorem \ref{est-Pysigma} it follows that
\beas
\int_{|x|^2+y^2< 1}|P_y^\sigma(x)|^q~dx~dy &\leq & C \int_{|x|^2+y^2< 1} y^{2\sigma q} (|x|^2+y^2)^{-nq/2-\sigma q} ~dx~dy\\
&\leq & C \int_{y=0}^1 \int_{\{H\in \overline{\mathfrak{a^+}}: |H|< 1\}}   y^{2\sigma q} (|H|^2+y^2)^{-nq/2-\sigma q} \, |H|^{n-l} ~dH~dy\\
&=& \int_{0}^1 \int_{0}^{1}y^{2\sigma q}(r^2+y^2)^{-nq/2-\sigma q}~r^{n-l}~r^{l-1}~dr~dy\\
&\leq& \int_0^1 \left(\int_0^{\infty} (1+s^2)^{-nq/2-\sigma q} s^{n-1} ~ds \right)~y^{n-nq}~dy.
\eeas
We  now use the following fact from \cite[3.251, (2); p.324]{GR}
\be \label{defn-beta}
\int_{0}^{\infty} x^{\mu-1}(1+x^2)^{\nu-1} ~dx= \frac{1}{2} B\left(\mu/2, (1-\nu-\mu/2)\right), \txt{ if } \Re \mu>0, \txt{ and } \Re(\nu+\mu/2)<1.
\ee
In our case, $\mu=n$ and $\nu= -nq/2-\sigma q+1$. Hence, $\nu+\mu/2<1$ if and only if $q> n/(n+2\sigma)$. Therefore, if $q> n/(n+2\sigma)$  the above integral reduces to 
\beas
\frac{1}{2} B\left(n/2, (nq/2 +\sigma q -n/2)\right)\int_{0}^1 y^{n-nq}~dy. 
\eeas
This is finite only if $q< (1+n)/n$. Hence, for $n/(n+2\sigma)<q<1+\frac 1n$,
\begin{equation}
\nonumber
\int_{|x|^2+y^2\leq 1}|P_y^\sigma(x)|^q~dx~dy<\infty.
\end{equation}
On the other hand for $q>1$, using the estimate of Jacobian in (\ref{J-est}) and the  asymptotic behaviour of $\phi_0$ given in (\ref{estimate-phi0}), it follows from  Theorem \ref{est-Pysigma} that
\beas &&\int_{|x|^2+y^2\geq 1}|P_y^\sigma(x)|^q~dx~dy \\
&&\leq \int_{|x^2|+y^2 \geq 1} \frac{y^{2\sigma q}}{\left(4^\sigma \Gamma(\sigma)\right)^q} \left(\sqrt{|x|^2+y^2}\right)^{-(l/2+\,|\Sigma_0^+|+\,\sigma+\,1/2)q}~e^{-|\rho|q\sqrt{|x|^2+y^2}}~|\phi_0(x)|^q~dx~dy\\
&&\leq C \int_{|x|^2+y^2 \geq 1} y^{2\sigma q} ~e^{-\frac{|\rho|(q+1)}{2}\sqrt{|x|^2+y^2}}~e^{-\frac{|\rho|(q-1)}{2}\sqrt{|x|^2+y^2}}~|\phi_0(x)|^q~dx~dy\\
&&\leq C \int_{\left\{(H, y)\in \overline{\mathfrak{a}^+}\times (0, \infty) : |H|^2+y^2 \geq 1\right\}} y^{2\sigma q}~e^{-\frac{|\rho|(q-1)|y|}{2}}~e^{-\frac{|\rho|(q+1) |H|}{2}} ~|H|^{|\Sigma_0^+|q}\, e^{-q\,\rho(H)}~e^{2\rho(H)}~dH~dy\\
&& \leq \left(\int_{0}^\infty y^{2\sigma q}~e^{-|\rho|(q-1)|y|/2}~dy\right) \left(\int_{\overline{\mathfrak{a^+}}} |H|^{|\Sigma_0^+| q}~e^{-\frac{3}{2}(q-1)\rho(H)}~dH\right)< \infty.
\eeas

This completes the proof.

\end{proof}

We are now in a position to prove Theorem \ref{mapping-property}.  We follow similar ideas which are used to the proof of \cite[Theorem B]{MOZ}.
\begin{proof}[Proof of Theorem \ref{mapping-property}] We first prove $(1)$. 
Let $u$ be the solution of (\ref{extn1}) with boundary value $f \in H^{\sigma}(X)$, and let 
\bes \mathcal{U}(\la, k, \eta) =\mathcal{F}\left( \widetilde u (\la, k)\right)(\eta),\: \textit{ for } \la\in \mathfrak{a}^\ast, k\in K, \eta\in {\R}_+ 
\ees 
be the composition of the Helgason and the Euclidean Fourier transform on $X\times \R$. Multiplying $y^2$ on both sides of the equation (\ref{extn1}) and taking the composition of Helgason and Euclidean Fourier transform on $X \times \R$ it follows that
\bes 
\frac{\partial^2}{\partial \eta^2} \left((|\lambda|^2+|\rho|^2+\eta^2) ~\mathcal{U}(\lambda, k, \eta)\right) - (1-2\sigma) \frac{\partial}{\partial \eta}\left(\eta~ \mathcal{U}(\lambda, k, \eta) \right)=0
\ees
which is equivalent to
\be \label{eqn-u-tilde} 
\left\{ (|\la|^2+|\rho|^2+\eta^2)\frac{\partial^2}{\partial^2 \eta} +(3+2\sigma)\eta\frac{\partial}{\partial \eta}+(1+2\sigma)\right\} \mathcal{U}(\la, k, \eta)=0.
\ee
Let $t= \frac{\eta}{\sqrt{|\la|^2+|\rho|^2}}$ and we define
\bes
v(\la, k, t)= \mathcal{U}(\la, k, \eta).
\ees
Then equation (\ref{eqn-u-tilde}) reduces to
\bes 
\mathcal D_{\sigma, t} v(\la, k, t):=   
\left\{(1+t^2)\frac{d^2}{d t^2} +(2\sigma + 3)t\frac{d}{dt}+ (2\sigma +1)\right\} v(\lambda, k, t)=0.
\ees
Since $f(x)= u(x, 0)$ for $x\in X$, by the Euclidean Fourier inversion formula we have
\bes
\widetilde{f}(\lambda, k)=u(\cdot, 0)^{\widetilde{}}(\la, k)= \frac{1}{\sqrt{2\pi}} \int_{\R} \mathcal{U}(\lambda, k, \eta)~d\eta= \frac{\sqrt{|\la|^2+|\rho|^2}}{\sqrt{2\pi}}\int_\R v(\la, k, t)~dt.
\ees
Therefore, the function $v$ satisfies 
\bes 
\mathcal D_{\sigma, t} v(\lambda, k, t)=0, \text{ and } \int_\R v(\lambda, k, t)\,dt=\frac{\sqrt{2\pi}}{\sqrt{|\lambda|^2 + |\rho|^2}} \widetilde{f}(\lambda, k),
\ees
for almost every $(\lambda, k)\in \mathfrak a^\ast \times K$. Hence, the function $v$ is given by 
\be \label{eqn-v-f}
v(\la, k, t)= \frac{\sqrt{2\pi}}{\sqrt{|\la|^2+|\rho|^2}} \widetilde f(\la, k)~\psi(t), \ee 
where $\psi$ satisfies 
\be \label{psi-pr}
\mathcal D_{\sigma, t}\psi=0, \:\: \txt{and}\:\: \int_{\R} \psi(t)~dt=1.
\ee
The equation $D_{\sigma,t} \psi=0$ has a fundamental system of solutions spanned by 
\beas
&& \psi_1(t)= \, _2F_1\left(\frac{1}{2}, \sigma +\frac{1}{2}; \frac{1}{2}; -t^2\right)=(1+t^2)^{-\sigma-1/2},\\
&& \psi_2(t)= t \, _2F_1\left(1, \sigma +1; \frac{3}{2}; -t^2\right).
\eeas
Using (\ref{eqn-v-f}) it is now easy to check that
\bea \label{eqn-iso}
&&\int_{\mathfrak a^\ast \times K \times \R}|\mathcal{U}(\la, k, \eta)|^2 (|\la|^2 + |\rho|^2 + \eta^2)^{\sigma+ \frac 12}~|{\bf c}(\la)|^{-2}~d\la~dk~d\eta \nonumber \\
&=& \int_{\mathfrak a^\ast \times K\times \R} |v(\lambda, k, t)|^2 ~(|\la|^2 + |\rho|^2)^{\sigma+1}~\left(1+t^2\right)^{\sigma+ \frac 12}~|{\bf c}(\la)|^{-2}~d\la~dk~dt \nonumber \\
&=&2\pi\int_{\mathfrak{a}^\ast \times K}|\widehat f(\la, k)|^2~(|\la|^2+|\rho^2|)^{\sigma}~|{\bf c}(\la)|^{-2}~d\la~dk \int_{\R} |\psi(t)|^2~(1+t^2)^{\sigma+ \frac 12}~dt.
\eea
Since $f\in H^\sigma$, it follows that $u\in H^{\sigma+ \frac 12}$ if and only if $\psi \in L^2(\R, (1+t^2)^{\sigma+ \frac 12}\,dt)$. It is easy to check from the asymptotic properties of hypergeometric function that $\psi_2\notin L^2(\R, (1+t^2)^{\sigma+ \frac 12}\,dt)$ (see \cite[Theorem 2.3.2]{Andrews}).  Hence, we choose $\psi(t)$ to be a constant multiple of $\psi_1(t)=(1+t^2)^{-\sigma- \frac 12}$. From (\ref{defn-beta}) we get that $\|\psi_1\|_{L^1(\R)}=\sqrt{\pi}\, \Gamma(\sigma)/\Gamma(\sigma+ \frac 12)$. Hence, using (\ref{psi-pr}) it follows that
\bes
\psi(t)= \frac{\Gamma(\sigma+1/2)}{\sqrt \pi \Gamma(\sigma)}\psi_1(t).
\ees
We now observe that 
\bes
\int_{\R} |\psi(t)|^2~(1+t^2)^{\sigma+ \frac 12}~dz= \frac{\Gamma(\sigma+1/2)}{\sqrt \pi \Gamma(\sigma)},
\ees
and hence from (\ref{eqn-iso}),
\beas
\|u\|_{H^{\sigma+ \frac 12}(X\times \R_+)}^2= \frac{2\sqrt{\pi} \Gamma(\sigma+ \frac 12)}{\Gamma(\sigma)} \|f\|_{H^{\sigma}(X)}.
\eeas
This completes the proof of part (1). We now prove part (2). We first observe that 
\bes
\|T_\sigma f\|_{L^q(X \times \R_+)}^q= \int_{0}^\infty \|f \ast P_y^\sigma\|_{L^q(X)}^q~dy.
\ees
Also, from Theorem \ref{est-Pysigma} it follows that for each $y>0$ the function $P_y^\sigma \in L^q(X)$, for all $q>1$. Therefore, by Kunze-Stein phenomenon (Remark \ref{rem-k-s}), for $1\leq p < q\leq 2$
\bes
\|f\ast P_y^\sigma\|_{L^q(X)}\leq C\|f\|_{L^p(X)}\|P_y^\sigma\|_{L^q(X)}.
\ees
 Therefore, by Lemma \ref{lem-Py-sigma-lq}  it follows that
\be \label{interpolation-1}
T_\sigma: L^p(X) \rightarrow L^q(X\times \R_+),
\ee
is a bounded map, for $1\leq p< q< (n+1)/n$. 
We also observe that 
\be\label{interpolation-2}
T_\sigma: L^\infty(X)\ra L^\infty(X\times \R_+),
\ee
is a bounded map, as
the integral  $\int_X P_y^\sigma(x)\,dx=1$ for all $y>0$. 
By Riesz Thorin interpolation theorem it now follows from (\ref{interpolation-1}) and (\ref{interpolation-2}) that
\be \label{interpolation-3}
T_\sigma: L^p(X) \rightarrow L^q(X\times \R_+),
\ee
is bounded for $1\leq p < \infty$ and $p<q<(\frac {n+1}{n})p$.
We now prove that
\bes
\|T_\sigma f\|_{L^{q}(X\times \R_+)}\leq C \|f\|_{L^p(X)},
\ees
for $p> 1$ and $q=(\frac{n+1}{n})p$. By (\ref{interpolation-2}) and Marcinkiewicz interpolation theorem
it is enough to show that 
\bes
T_\sigma: L^1(X) \ra L^{(1+n)/n, \infty}(X\times \R_+).
\ees 
Using Theorem \ref{est-Pysigma} and the boundedness of the function $\phi_0$ we get that
\beas
&& |T_\sigma f(x, y)|  \leq \int_{X} |f(z)| P_y^\sigma (z^{-1}x)~|f(z)|~dz \leq Cy^{-n} ~\|f\|_{L^1(X)}\\
&& +  Cy^{2\sigma}\int_{|z^{-1}x|^2+y^2 \geq 1} \sqrt{(|z^{-1}x|^2+y^2)}^{(-l/2-1/2-\sigma-|\Sigma_0^{+}|)}~ e^{-|\rho| \sqrt{(|z^{-1}x|^2+y^2)}|} |f(z)|~dz\\
&&\leq Cy^{-n} ~\|f\|_{L^1(X)} + C y^{-n}~ \|f\|_{L^1(X)}~ \sup{y\in \R_+} \left(y^{2\sigma +n} e^{-|\rho| y}\right)\\
&&\leq  C y^{-n}~\|f\|_{L^1(X)}.
\eeas
Hence, $|T_a f(x, y)|> \la$ implies that $y\leq \left(\frac{C \|f\|_{L^1(X)}}{\la}\right)^{\frac{1}{n}}=b$ (say).
Then Chebyshev's inequality yields
\beas
&& m\left(\{(x, y)\in X\times \R_+: |T_af(x, y)| > \la\}\right) \\
&=& m\left(\{(x, y)\in X\times \R_+: y< b, |T_af(x, y)| > \la\}\right)\\
&\leq& \frac{1}{\la}\int_{\{(x, y)\in X\times \R_+: y< b\}} |T_af(x, y)|~dx~dy\\
&\leq& \frac{C_a}{\la} \int_{X} |f(z)|\int_{\{(x, y)\in X\times \R_+: y< b\}} P_y^\sigma( z^{-1}x)~dx~dy~dz\\
&\leq& \frac{C_\sigma}{\la} \|f\|_{L^1(X)} \,b= C_\sigma \left(\frac{\|f\|_{L^1(X)}}{\la}\right)^{1+\frac{1}{n}}.
\eeas
The last inequality follows because of the fact that  $\int_X P_y^\sigma(x)\,dx=1$ for all $y>0$.
This completes the proof.
\end{proof}

\section{Expression of the kernel $P_y^\sigma$}
In the case of $\R^n$ and of the Heisenberg groups the function $P_y^\sigma$ is the classical Poisson kernel. In the case of symmetric spaces, we only have the integral expression as in Theorem \ref{soln-extn-possion} and the both-sides estimates (Theorem \ref{est-Pysigma}) for $P_y^\sigma$. In this section we write the precise expression of $P_y^\sigma$ for complex and rank one symmetric spaces using the expression of the heat kernel.

\subsection{$G$ is complex} In this case we have the following formula for the heat kernel \cite{AO}
\bes
h_t(\exp H)=(4\pi t)^{-n/2}~e^{-|\rho|^2 t}~\left(\prod_{\alpha\in \Sigma^+}\frac{\alpha(H)}{\sinh \alpha(H)}\right)~e^{-H^2/4t}, \:\: t>0, ~H\in \mathfrak a.
\ees
It now follows from the definition (\ref{defn-P}) of $P_y^\sigma$  
\beas
&&P_y^\sigma(\exp H)=\frac{y^{2\sigma}}{4^\sigma \Gamma(\sigma)} (4\pi t)^{-n/2}~\left(\prod_{\alpha\in \Sigma^+}\frac{\alpha(H)}{\sinh \alpha(H)}\right)~\int_{0}^{\infty} t^{-n/2-\sigma-1}~e^{-|\rho|^2 t}~e^{-\left(|H|^2+y^2\right)/4t}~\frac{dt}{t^{1+\sigma}}\\
&=& \frac{y^{2\sigma}}{\Gamma(\sigma)}~2^{1-n/2-\sigma}~\pi^{-n/2}~\left(\prod_{\alpha\in \Sigma^+}\frac{\alpha(H)}{\sinh \alpha(H)}\right)~\left(\frac{\sqrt{|H|^2+y^2}}{|\rho|}\right)^{-(n+2\sigma)/2}~K_{-n/2-\sigma}(\sqrt{|H|^2+y^2}|\rho|).
\eeas
Here the last equality follows from the formula \cite[3.471(9), p. 368]{GR}, and  $K_{-n/2-\sigma}$ is the modified Bessel function (defined in \cite[8.407 (1), p. 911]{GR}).
\subsection{$X$ is of rank one}
Let $F = \R, \C, H$, or $O$ be the real numbers, the complex numbers, the quaternions or the
Cayley octonions respectively. The rank one symmetric spaces can be realized as the hyperbolic space $\mathbb H^n(F)$. Here the subscript $n$ denotes the dimension over the base field $F$. Using the expression of the heat kernel \cite{AO, GM} we have the following results. 
\begin{enumerate}
\item  $X=\mathbb H^n(\R)$, and $n\geq 3$ odd. Using the formula \cite[3.471(9), p. 368]{GR} we get
\beas
P_y^\sigma(x) &=& c\int_{0}^\infty t^{-1/2}~e^{-\rho^2 t}~e^{-y^2/4t} \left(-\frac{1}{\sinh x}~\frac{\partial}{\partial x}\right)^{(n-1)/2}~e^{-|x|^2/4t}~\frac{dt}{t^{1+\sigma}}\\
&=& c\left(-\frac{1}{\sinh x}~\frac{\partial}{\partial x}\right)^{(n-1)/2} \int_{0}^\infty t^{-3/2-\sigma}~e^{-\rho^2t }~e^{-(|x|^2+y^2)/4t}~dt\\
&=& c  \left(-\frac{1}{\sinh x}~\frac{\partial}{\partial x}\right)^{(n-1)/2} \left(\frac{\sqrt{|x|^2+y^2}}{\rho}\right)^{-\sigma-1/2}~K_{-\sigma-1/2}(\rho\sqrt{|x|^2+y^2}).
\eeas

\item $X=\mathbb H^n(\R)$, and $n\geq 2$ even. Using the formula \cite[3.471(9), p. 368]{GR} we get
\beas 
&&P_y^\sigma(x) \\
&=&c \int_{0}^\infty t^{-1/2}~e^{-\rho^2 t}~e^{-y^2/4t}~ \int_{x}^\infty \frac{\sinh z}{\sqrt{\cosh^2 z- \cosh^2 x}}\left(-\frac{1}{\sinh z}~\frac{\partial}{\partial z}\right)^{n/2}~e^{-|z|^2/4t}~dz~\frac{dt}{t^{1+\sigma}}\\
&=&c \int_{x}^\infty \frac{\sinh z}{\sqrt{\cosh^2 z- \cosh^2 x}}\left(-\frac{1}{ \sinh z}~\frac{\partial}{\partial z}\right)^{n/2}~\int_{0}^\infty t^{-3/2-\sigma}~e^{-\rho^2 t}~e^{-\left(|z|^2+y^2\right)/4t}~dt~dz\\ 
&=&c\int_{x}^\infty \frac{\sinh z}{\sqrt{\cosh^2 z- \cosh^2 x}}\left(-\frac{1}{\sinh z}~\frac{\partial}{\partial z}\right)^{n/2}\left(\frac{\sqrt{|z|^2+y^2}}{\rho}\right)^{-\sigma-1/2}~K_{-\sigma-1/2}(\rho\sqrt{|z|^2+y^2})~dz.
\eeas

\item $X=\mathbb H^n(F)$ where $F = \C, H$ or $O$. Then there exist constants $c_1, c_2, \cdots, c_{n/2}$ such that
\beas
P_y^\sigma(x) &=& \int_{0}^\infty t^{-1/2}~e^{-\rho^2 t}\sum_{j=1}^{n/2} c_j \int_{x}^\infty \frac{\sinh z}{\sqrt{\cosh^2 z- \sinh^2 x}}~(\cosh z)^{j+1-d}~\left(-\frac{1}{2\pi \sinh z}~\frac{\partial}{\partial z}\right)^{j+m_\alpha/2}\\
&& e^{-|z|^2/4t}~dz~\frac{dt}{t^{1+\sigma}}\\
&=& c_\sigma \sum_{j=1}^{d/2} c_j \int_{x}^\infty \frac{\sinh z}{\sqrt{\cosh^2 z- \sinh^2 x}}~(\cosh z)^{j+1-d}~\rho^{1+2\sigma}~\left(-\frac{1}{2\pi \sinh z}~\frac{\partial}{\partial z}\right)^{j+m_\alpha/2}\\
&& 2 \left(\frac{2}{\sqrt{|z|^2+y^2} \rho}\right)^{\sigma+1/2} ~K_{-\sigma-1/2}(\rho\sqrt{|z|^2+y^2})~dz,
\eeas
where the constant $c_\sigma$ depends only on $\sigma$.
\end{enumerate}

\section{Poincar\'{e}-Sobolev inequality}
In this section we prove Theorem \ref{thm-p-s-X}. For the convenience of the reader we restate the theorem here. 
\begin{thm}
Let $\dim X=n\geq 3$ and $0<\sigma<\min\{l+ 2|\Sigma_0^+|, n\}$. Then for $2< p\leq \frac{2n}{n-\sigma}$ there exists $S=S_{n, \sigma, p}>0$ such that for all $f\in H^{\frac{\sigma}{2}}(X)$ 
\be\label{est-poin}
\|(-\Delta-|\rho|^2)^{\sigma/4}f\|_{L^2(X)}^2 \geq S\|f\|_{L^p(X)}^2.
\ee
\end{thm}
\begin{proof}
We first observe that it is enough to prove the result for $f\in C_c^\infty(X)$. It also suffices to show that 
\be \label{ets1}
\int_X f(x) ~(-\Delta-|\rho|^2)^{-\sigma/2}f(x)~dx \leq C \|f\|_{L^{p'}(X)}^2.
\ee
Indeed, if (\ref{ets1}) holds, then by H\"older's inequality  
\beas
\left|\left\langle f, g \right\rangle \right| &=& \left|\left\langle \left(-\Delta-|\rho|^2\right)^{\sigma/4}f, \left(-\Delta-|\rho|^2\right)^{-\sigma/4}g\right\rangle\right|\\  \\
&\leq &\left \|\left(-\Delta-|\rho|^2\right)^{\sigma/4}f\right\|_{L^2(X)} \,\,\left\|\left(-\Delta-|\rho|^2\right)^{-\sigma/4}g\right\|_{L^2(X)} \\ \\
&=& \left\langle \left(-\Delta-|\rho|^2\right)^{\sigma/2}f, f\right\rangle^{1/2} \left\langle \left(-\Delta-|\rho|^2\right)^{-\sigma/2}g, g\right\rangle^{1/2} \\ \\
&\leq & C^{\frac{1}{2}} \left\langle \left(-\Delta-|\rho|^2\right)^{\sigma/2}f, f\right\rangle^{\frac{1}{2}} \|g\|_{L^{p'}(X)},
\eeas	
and hence
\bes
\|f\|_{L^p(X)} \leq C_n^{\frac{1}{2}} \left\langle \left(-\Delta-|\rho|^2\right)^{\sigma/2}f, f\right\rangle^{\frac{1}{2}}.
\ees
We now prove (\ref{ets1}).
Let $k_\sigma$ be the Schwartz kernel for the operator $(-\Delta- |\rho|^2)^{-\sigma/2}$. We have the following well-known estimates due to Anker and Ji \cite[Theorem 4.2.2]{AJ}, for $0< \sigma< l+2|\Sigma_0^+|$ 
\bea \label{K-infty-est}
k_\sigma(x)&&\asymp |x|^{\sigma-l-2|\Sigma_0^+|}~\phi_0(x),\:\: |x|\geq 1,\\
&& \asymp |x|^{\sigma-n},\:\: |x|< 1 \nonumber.
\eea
To prove (\ref{ets1}), it is enough to show that
\bes
\|f \ast k_\sigma\|_{L^p(X)}\leq C \|f\|_{L^{p'}(X)}.
\ees	
Let $\chi$ be the characteristic function of the unit ball ${\bf B}(o, 1)$ and $k_\sigma^0(x)= \chi(x) ~ k_\sigma(x)$ and $k_\sigma^\infty=k-k_0$. Now, by Young's inequality we have that  
\bes
\|f\ast k_\sigma^0\|_{L^p(X)}\leq C\|f\|_{L^{p'}(X)}~ \|k_\sigma^0\|_{L^{p/2}(X)},
\ees
and
\beas
\|k_\sigma^0\|_{L^{p/2}(X)}^{\frac{p}{2}}\asymp \int_{0}^{1}|t|^{(\sigma-n)p/2}~|t|^{|\Sigma^+|}~t^{l-1}~dt.
\eeas
The right-hand side is finite if
$p<\frac{2n}{n-\sigma}$. Using the fact that for $r<1$, the volume of the ball $B(o, r)$ in $X$ is of order $r^{n}$,  it is easy to check that $k_\sigma^0\in L^{\frac{n}{n-\sigma}, \infty}(X)$. By Young's inequality for weak type spaces \cite[Theorem 1.4.24. page 63]{G} it follows that
\bes
\|f \ast k_\sigma^0\|_{L^{\frac{2n}{n-\sigma}}(X)} \leq C \|f\|_{L^{\frac{2n}{n+\sigma}}(X)}.
\ees	
Therefore, we have for all $p\leq \frac{2n}{n-\sigma}$,
\begin{equation}\label{around-origin}
\|f \ast k_\sigma^0\|_{L^p(X)} \leq C \|f\|_{L^{p'}(X)}.
\end{equation}
Next, we shall show that for $p>2$,
\begin{equation}
\nonumber
\|f \ast k_\sigma^\infty\|_{L^p(X)} \leq C_p \|f\|_{L^{p'}(X)}.
\end{equation}
To prove this we shall use complex interpolation theorem  and the idea of \cite[Theorem 4.1]{RS}. For $\Re z\geq  -\frac 12$, we define an analytic family of linear operators $T_z$ from $(X, dx)$ to itself as follows:
\begin{equation}
\nonumber
T_z f= f\ast (k_\sigma^\infty)^{1+z}.
\end{equation}
For $z=-\frac{1}{2} + iy$, we have
\beas
\|T_z f\|_{L^\infty(X)} &=&\|f\ast (k_\sigma^\infty)^{\frac{1}{2} +iy}\|_{L^\infty(X)} \\
&\leq & C \, \sup_{\{x\in X:|x|\geq 1\}}\varphi_0(x)^{\frac{1}{2}} |x|^{(\sigma-l)/2-|\Sigma_0^{+}| } \, \|f\|_{L^1(X)}\\
&\leq& C \|f\|_{L^1(X)}.
\eeas
For $z=\epsilon + iy, \epsilon>0$, we have
\beas
\|T_z f\|_{L^2(X)}^2 &=& \int_\R \int_K \left|\widetilde{f}(\lambda, k)\right|^2 \left|\widehat{ (k_\sigma^\infty)^{1+\epsilon +iy}}(\lambda)\right|^2 |{\bf c}(\lambda)|^{-2} \,d\lambda~dk\\
&\leq& \sup \left|\widehat{ (k_\sigma^\infty)^{1+\epsilon +iy}}(\lambda)\right|^2 \|f\|_{L^2(X)}^2.
\eeas
Now, by Theorem \ref{thm-phi} it follows that for $\lambda\in \mathfrak a^\ast$ and
$\epsilon > 0$  
\beas
|\widehat{ (k_\sigma^\infty)^{1+\epsilon +iy}}(\lambda)|&\leq& \left|\int_{\{x\in X: |x|\geq 1\}}  |x|^{(\sigma-l-2|\Sigma_0^+|)(1+\epsilon+iy)}~(\phi_0(x))^{1+\epsilon+iy}~\phi_{-\lambda}(x)~dx\right|\\
&\leq& \int_{\{x\in X: |x|\geq 1\}} \phi_{0}(x)^{2+\epsilon}~ dx< \infty,
\eeas
and hence $\|T_zf\|_{L^\infty(X)}\leq \|f\|_{L^2(X)}$.
Hence, by analytic interpolation for $p>2$,
\be \label{around-infty}
\|f \ast k_\sigma^\infty\|_{L^p(X)}= \|T_0 f\|_{L^p(X)} \leq C \|f\|_{L^{p'}(X)}.
\ee
Therefore, from (\ref{around-origin}) and from (\ref{around-infty}), it follows that for all $2<p\leq \frac{2n}{n-\sigma}$,
\begin{equation}
\nonumber
\|f \ast k_\sigma\|_{L^p(X)} \leq C \|f\|_{L^{p'}(X)}.
\end{equation}
This completes the proof.

\end{proof}	

As a corollary of the theorem above we have the following
\begin{cor}
Let $2< p\leq \frac{2n}{n-2}$ and $\dim X=n\geq 3$. Then there exists $S_{n,p}>0$ such that for all $u\in H^1(X)$, 
\bes \|\nabla u\|_{L^2(X)}^2-|\rho|^2\|u\|_{L^2(X)}^2\geq S_{n,p} \|u\|_{L^p(X)}^2.
\ees
\end{cor}

\noindent{\bf Acknowledgement:} 
The first author is supported by the Post Doctoral fellowship from IIT Bombay. The second author is
supported partially by SERB, MATRICS, MTR/2017/000235. The authors are thankful to  Swagato K Ray for numerous useful discussions and detailed comments. The authors are also grateful to Sundaram Thangavelu for valuable suggestions for the improvement of the paper.

\end{document}